\documentclass[reqno,12pt]{article}
\usepackage{a4wide,color,eucal,enumerate,mathrsfs}
\usepackage[normalem]{ulem}
\usepackage{amsmath,amssymb,epsfig,bbm,graphicx}
\numberwithin{equation}{section}
\usepackage{tikz,tikz-cd}

\usepackage[latin1]{inputenc}
\usepackage{hyperref}
\usepackage{pdfsync}

\newcommand{\Md}{{\rm Mod}_{p,\smm}}
\newcommand{\Mdc}{{\rm Mod}_{p,\smm,c}}

\newcommand{\N}{\mathbb{N}}

\newcommand{\R}{\mathbb{R}}

\renewcommand{\AA}{\mathscr{A}}
\newcommand{\BB}{\mathscr{B}}

\newcommand{\FF}{\mathscr{F}}

\newcommand{\cB}{{\ensuremath{\mathcal B}}}

\newcommand{\cE}{{\ensuremath{\mathcal E}}}
\newcommand{\cF}{{\ensuremath{\mathcal F}}}

\newcommand{\cM}{{\ensuremath{\mathcal M}}}

\newcommand{\cP}{{\ensuremath{\mathcal P}}}

\newcommand{\mm}{{\mbox{\boldmath$m$}}}

\newcommand{\smm}{{\mbox{\scriptsize\boldmath$m$}}}

\newcommand{\eeta}{{\mbox{\boldmath$\eta$}}}

\newcommand{\ppi}{{\mbox{\boldmath$\pi$}}}
\newcommand{\rrho}{{\mbox{\boldmath$\rho$}}}
\newcommand{\ssigma}{{\mbox{\boldmath$\sigma$}}}

\newcommand{\seeta}{{\mbox{\scriptsize\boldmath$\eta$}}}

\newcommand{\sfd}{{\sf d}}

\newcommand{\sfi}{{\sf i}}

\newcommand{\sfr}{{\sf r}}
\newcommand{\sfs}{{\sf s}}
\newcommand{\sft}{{\sf t}}

\newcommand{\rmC}{{\mathrm C}}

\newcommand{\Kliminf}{K\kern-3pt-\kern-2pt\mathop{\rm lim\,inf}\limits}  % Kuratowski liminf di insiemi
\newcommand{\Lip}{\mathop{\rm Lip}\nolimits}          %Lipschitz constant
\renewcommand{\d}{{\mathrm d}}

\newcommand{\restr}[1]{\lower3pt\hbox{$|_{#1}$}}

\newcommand{\Haus}[1]{{\mathscr H}^{#1}}     % Misura di Hausdorff
\newcommand{\Leb}[1]{{\mathscr L}^{#1}}      % Misura di Lebesgue
\newcommand{\eps}{\varepsilon}  
\newcommand{\nchi}{{\raise.3ex\hbox{$\chi$}}}

\newcommand{\BorelSets}[1]{\BB(#1)}

\newcommand{\Probabilities}[1]{\mathscr P(#1)}          % misure di probabilita'
\newenvironment{proof}{\removelastskip\par\medskip   % inizio e fine dimostrazione
\noindent{\em Proof.}
\rm}{\penalty-20\null\hfill$\square$\par\medbreak}

\newtheorem{theorem}{Theorem}[section]

\newtheorem{corollary}[theorem]{Corollary}
\newtheorem{lemma}[theorem]{Lemma}
\newtheorem{proposition}[theorem]{Proposition}
\newtheorem{definition}[theorem]{Definition}

\newtheorem{example}[theorem]{Example}

\newtheorem{remark}[theorem]{Remark}
\newcommand{\e}{{\rm{e}}}                           % mappa di valutazione, bisogna mettere `a mano' il tempo t
\renewcommand{\mm}{\mathfrak m}
\renewcommand{\smm}{{\mbox{\scriptsize$\mm$}}}

\newcommand{\esssup}{\mathop{\rm esssup}}
\newcommand{\ep}{\varepsilon}
\newcommand{\Ldp}{\mathcal{L}^p_+(X, \mm)}
\newcommand{\Curvesnonpara}[1]{\mathscr C(#1)}

\newcommand{\res}{\mathop{\hbox{\vrule height 7pt width .5pt depth 0pt
\vrule height .5pt width 6pt depth 0pt}}\nolimits}

\renewcommand{\Probabilities}[1]{\mathcal P(#1)}
\newcommand{\KK}{\mathscr K}
\newcommand{\AC}{\mathrm{AC}}

\newcommand{\Lipc}[1]{\AC^\infty_c([0,1];#1)}
\newcommand{\sfj}{\mathsf j}
\newcommand{\sfk}{\mathsf k}
\newcommand{\J}I
\newcommand{\cI}{{\ensuremath{\mathcal I}}}
\newcommand{\cII}{\mathfrak{I}}
\newcommand{\pgamma}{{\underline{\gamma}}}

\title{On the duality between $p$-modulus\\ and probability measures}

\begin{document}
\author{L.~Ambrosio\thanks{Scuola Normale Superiore, Pisa. email: \textsf{luigi.ambrosio@sns.it}},
Simone Di Marino\thanks{Scuola Normale Superiore, Pisa. email: \textsf{simone.dimarino@sns.it}},
Giuseppe Savar\'e\thanks{Dipartimento di Matematica, Universit\`a di Pavia. email: \textsf{giuseppe.savare@unipv.it}}
}
\maketitle

\begin{abstract} Motivated by recent developments on calculus in metric measure spaces $(X,\sfd,\mm)$, we prove a general
duality principle between Fuglede's notion \cite{Fuglede} of $p$-modulus for families of finite Borel measures in $(X,\sfd)$ and probability measures
with barycenter in $L^q(X,\mm)$, with $q$ dual exponent of $p\in (1,\infty)$. We apply this general duality principle to study null sets for families
of parametric and non-parametric curves in $X$. In the final part of the paper we provide a new proof, independent of optimal transportation, 
of the equivalence of notions of weak upper  gradient based on $p$-podulus (\cite{Koskela-MacManus}, \cite{Shanmugalingam00}) and 
suitable probability measures in the space of curves (\cite{Ambrosio-Gigli-Savare11}, \cite{Ambrosio-Gigli-Savare12}).
\end{abstract}

\tableofcontents

\section{Introduction}

The notion of $p$-modulus ${\rm Mod}_p(\Gamma)$ for a family $\Gamma$ of curves has been introduced by Beurling and Ahlfors in \cite{Ahlfors} and then
it has been deeply studied by Fuglede in \cite{Fuglede}, who realized its significance in Real Analysis and proved that Sobolev $W^{1,p}$ functions
$f$ in $\R^n$ have representatives $\tilde{f}$ that satisfy
$$
\tilde f(\gamma_b)-\tilde f(\gamma_a)=\int_a^b \langle\nabla
f(\gamma_t),\gamma_t'\rangle \,\d t
$$
for ${\rm Mod}_p$-almost every absolutely continuous curve $\gamma:[a,b]\to\R^n$. Recall that if $\Gamma$ is a family of absolutely continuous
curves, ${\rm Mod}_p(\Gamma)$ is defined by
\begin{equation}\label{eq:modulus_euclideo}
{\rm Mod}_p(\Gamma):=\inf\left\{\int_{\R^n} f^p\, \d x:\ \text{$f:\R^n\to [0,\infty]$ Borel, $\int_\gamma f\geq 1$ for all $\gamma\in\Gamma$}\right\}. 
\end{equation}
It is obvious that this definition (as the notion of length) is parametric-free, because the curves are involved in the definition only
through the curvilinear integral $\int_\gamma f$. Furthermore, if
$\gamma:I\to X$, writing the curvilinear integral as $\int_I
f(\gamma_t)|\dot\gamma_t| \, \d t$,
with $|\dot\gamma|$ equal to the metric derivative, 
one realizes immediately that this notion makes sense for absolutely continuous curves in a general metric space $(X,\sfd)$, if we add a reference measure
$\mm$ to minimize the integral $\int f^p\,\d\mm$. The notion, denoted by ${\rm Mod}_{p,\mm}(\cdot)$, 
actually extends to families of continuous curves with finite length, which do have a Lipschitz 
reparameterization. As in \cite{Fuglede}, one can even go a step further, realizing
that the curvilinear integral in \eqref{eq:modulus_euclideo} can be written as 
$$
\int_X f \,\d J\gamma,
$$ 
where $J\gamma$ is  a positive finite measure in $X$, the image under $\gamma$ of the measure
$|\dot\gamma|\Leb{1}\res I$, namely
\begin{equation}\label{eq:Jintro}
J\gamma(B)=\int_{\gamma^{-1}(B)}|\dot\gamma_t|\,\d t\qquad\forall B\in\BorelSets{X}
\end{equation} 
(here $\Leb{1}\res I$ stands for the Lebesgue measure on $I$). It follows that one can define in a similar
way the notion of $p$-modulus for families of measures in $X$.

In more recent times, Koskela-Mac Manus \cite{Koskela-MacManus} and then Shanmugalingham \cite{Shanmugalingam00} used the
$p$-modulus to define the notion
of $p$-weak upper gradient for a function $f$, namely Borel functions $g:X\to [0,\infty]$ such that the upper gradient
inequality
\begin{equation}\label{eq:Fugledeabs}
|f(\gamma_b)-f(\gamma_a)|\leq\int_\gamma g
\end{equation}
holds along ${\rm Mod}_{p,\smm}$-almost every absolutely continuous curve $\gamma:[a,b]\to X$. 
This approach leads to a very successful Sobolev space theory in metric measure spaces $(X,\sfd,\mm)$, see for
instance \cite{Heinonen07,Bjorn-Bjorn11} for a very nice account of it. 

Even more recently, the first and third author and Nicola Gigli introduced (first in \cite{Ambrosio-Gigli-Savare11} for $p=2$, and then
in \cite{Ambrosio-Gigli-Savare12} for general $p$) another notion of weak upper gradient, based on suitable
classes of probability measures on curves, described more in detail in the final section of this paper. Since the axiomatization in \cite{Ambrosio-Gigli-Savare11}
is quite different and sensitive to parameterization, it is a surprising fact that the two approaches lead essentially to the same Sobolev
space theory (see Remark~5.12 of \cite{Ambrosio-Gigli-Savare11} for a more detailed discussion, also in connection with Cheeger's 
approach \cite{Cheeger00}, and Section~9 of this paper). We say
essentially because, strictly speaking, the axiomatization of \cite{Ambrosio-Gigli-Savare11} is invariant (unlike Fuglede's approach) under modification
of $f$ in $\mm$-negligible sets and thus provides only Sobolev regularity and not absolute continuity along almost every curve; however, choosing properly 
representatives in the Lebesgue equivalence class, the two Sobolev spaces can be identified.

Actually, as illustrated in \cite{Ambrosio-Gigli-Savare11}, \cite{Ambrosio-Gigli-Savare11bis}, \cite{Gigli12} (see also
the more recent work \cite{Bate}, in connection with Rademacher's theorem and Cheeger's Lipschitz charts), differential
calculus and suitable notions of {\it tangent} bundle in metric measure spaces can be developed in a quite natural way
using probability measures in the space of absolutely continuous curves.

With the goal of understanding deeper connections between the ${\rm Mod}_{p,\smm}$ and the probabilistic 
approaches, we show in this paper that the theory of $p$-modulus has a ``dual'' point of view,
based on suitable probability measures $\ppi$ in the space of curves;  the main difference with respect to \cite{Ambrosio-Gigli-Savare11} is that,
as it should be, the curves here are non-parametric, namely $\ppi$ should be rather thought as measures in a quotient space of curves. 
Actually, this and other technical aspects (also relative to tightness, since much better compactness properties are available at the level
of measures) are simplified if we consider $p$-modulus of families of
measures in $\mathcal{M}_+(X)$ (the space of all nonnegative and finite Borel measures on $X$), 
rather than $p$-modulus of families of curves:
if we have a family $\Gamma$ of curves, we can consider the family $\Sigma=J(\Gamma)$ and derive a representation formula
for ${\rm Mod}_{p,\smm}(\Gamma)$, see Section~\ref{sec:weakgrad}.
Correspondingly, $\ppi$ will be a measure on the Borel subsets of
$\mathcal{M}_+(X)$. 

For this reason, in Part I of this paper we investigate the duality at this level of generality, considering a family $\Sigma$ of measures in $\mathcal{M}_+(X)$.
Assuming only that $(X,\sfd)$ is complete and separable and $\mm$ is finite, we
prove in Theorem~\ref{tmain} that for all Borel sets
$\Sigma\subset\mathcal{M}_+(X)$ (and actually in the more general class of
Souslin sets) the following duality formula holds:
\begin{equation}\label{eq:mainresult}
\bigl[{\rm Mod}_{p,\smm}(\Sigma)\bigr]^{1/p}=\sup_{\seeta}\frac{\eeta(\Sigma)}{c_q(\eeta)}=
\sup_{\seeta(\Sigma)=1}\frac 1{c_q(\eeta)},
\qquad
\frac 1p+\frac 1q=1.
\end{equation}
Here the supremum in the right hand side runs in the class of Borel probability measures $\eeta$ in $\mathcal{M}_+(X)$ with
barycenter in $L^q(X,\mm)$, so that
$$
\text{there exists $g\in L^q(X,\mm)$ s.t.}\quad
\int \mu(A)\,\d\eeta(\mu)=\int_A g\,\d\mm\qquad\forall A\in\BorelSets{X};
$$
the constant $c_q(\eeta)$ is then defined as the $L^{q}(X,\mm)$ norm of the
``barycenter'' $g$. A byproduct of our proof is the fact that ${\rm Mod}_{p,\smm}$ is a Choquet capacity
in $\mathcal{M}_+(X)$, see Theorem~\ref{tmain}.
In addition, we can prove in Corollary~\ref{cor:cnes} existence of maximizers in \eqref{eq:mainresult} and obtain out of this necessary and sufficient
optimality conditions, both for $\eeta$ and for the minimal $f$ involved in the definition of $p$-modulus analogous
to \eqref{eq:modulus_euclideo}. See also Remark~\ref{rem:saturated}
for a simple application of these optimality conditions involving pairs $(\mu,f)$ on which the constraint is saturated, namely
$\int_X f\,\d\mu=1$. 

We are not aware of other representation formulas for ${\rm Mod}_{p,\smm}$, except in special cases: for instance in the case of
the family $\Gamma$ of curves connecting two disjoint compact sets $K_0$, $K_1$ of $\R^n$, the modulus in \eqref{eq:modulus_euclideo} 
equals (see \cite{Ziemer} and also \cite{Shanmugalingham01} for the extension to metric measure spaces, as well as
\cite{Shanmugalingham10} for related results) the capacity
$$
C_p(K_0,K_1):=\inf\left\{\int_{\R^n} |\nabla u|^p\,\d x: \text{$u\equiv 0$ on $K_0$, $u\equiv 1$ on $K_1$}\right\}.
$$
In the conformal case $p=n$, it can be also proved that $C_n(K_0,K_1)^{-1/(n-1)}$ equals ${\rm Mod}_{n/(n-1)}(\Sigma)$, where 
$\Sigma$ is the family of the Hausdorff measures $\Haus{n-1}\res S$, with $S$ separating $K_0$ from $K_1$ (see \cite{Ziemer1}). 

In the second part of the paper, after introducing in Section~6 the relevant space of curves $\AC^q([0,1];X)$ and a suitable
quotient space $\Curvesnonpara X$ of non-parametric nonconstant curves, we show how the basic duality result of Part I 
can be read in terms of measures and moduli in spaces of curves. For non-parametric curves this is accomplished
in Section~7, mapping curves in $X$ to measures in $X$ with the canonical map $J$ in \eqref{eq:Jintro}; in this case, the condition
of having a barycenter in $L^q(X,\mm)$ becomes
\begin{equation}\label{eq:nonparabari}
\biggl|\int \int_0^1 f(\gamma_t)|\dot\gamma_t|\,\d t \,\d\ppi(\gamma)\biggr|\leq C\|f\|_{L^p(X,\mm)}\qquad
\forall f\in \rmC_b(X).
\end{equation}
Section 8 is devoted instead to the case of parametric curves, where the relevant map curves-to-measures is
$$
M\gamma (B):=\Leb{1}(\gamma^{-1}(B))\qquad\forall B\in\BorelSets{X}.
$$
In this case the condition
of having a \emph{parametric} barycenter in $L^q(X,\mm)$ becomes
\begin{equation}\label{eq:parabari}
\biggl|\int \int_0^1 f(\gamma_t)\,\d t \,\d\ppi(\gamma)\biggr|\leq C\|f\|_{L^p(X,\mm)}\qquad
\forall f\in \rmC_b(X).
\end{equation}
The parametric barycenter can of course be affected by reparameterizations; a key result, stated in Theorem~\ref{thm:repara},
shows that suitable reparameterizations improve the parametric barycenter from $L^q(X,\mm)$
to $L^\infty(X,\mm)$. Then, in Section~9 we discuss the notion of null set of curves according to
 \cite{Ambrosio-Gigli-Savare11} and \cite{Ambrosio-Gigli-Savare12}
(where \eqref{eq:parabari} is strengthened by requiring $\bigl|\int f(\gamma_t)\,\d\ppi(\gamma)\bigr|\leq C\|f\|_{L^1(X,\mm)}$
for all $t$, for some $C$ independent of $t$)
and, under suitable invariance and stability assumptions on the set of curves, we compare this notion
with the one based on $p$-modulus. Eventually, in Section~10 we use there results to
prove that if 
a Borel function $f:X\to \R$ has a continuous representative along
a collection $\Gamma$ of the set $\mathrm{AC}^\infty([0,1];X)$ of the 
Lipschitz parametric curves with
$\mathrm{Mod}_{p,\mm}\bigl(M(\mathrm{AC}^\infty([0,1];X)\setminus\Gamma)\bigr)=0$,
then it is possible to find a distinguished 
$\mm$-measurable representative 
$\tilde f$ such that $\mm(\{f\neq \tilde f\})=0$ and
$\tilde f$ is absolutely continuous along
$\mathrm{Mod}_{p,\mm}$-a.e.-nonparametric curve.
By using these results to provide a more
direct proof of the equivalence of the two above mentioned notions of weak upper gradient, 
where different notions of null sets of curves are used to quantify exceptions to \eqref{eq:Fugledeabs}.

For the reader's convenience we collect in the next table and figure the main notation used, mostly
in the second part of the paper.

\medskip
\centerline{Main notation}
\smallskip
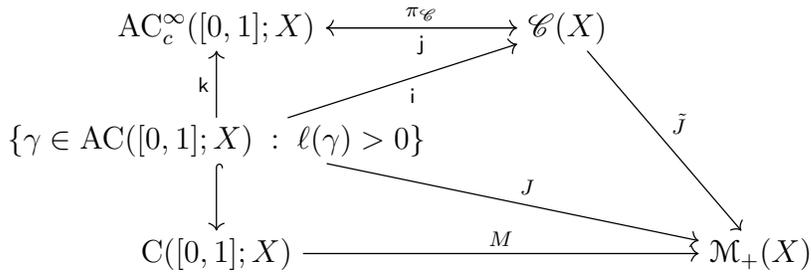
\halign{$#$\hfil\qquad &#\hfil\cr
{\cal L}^p_+(X,\mm) & Borel nonnegative functions $f:X\to [0,\infty]$ with $\int_X f^p\,\d\mm<\infty$\cr
L^p(X,\mm) & Lebesgue space of $p$-summable $\mm$-measurable functions\cr
\ell(\gamma) & Length of a parametric curve $\gamma$\cr 
\AC^q([0,1];X) & Space of parametric curves $\gamma:[0,1]\to X$ with $q$-integrable metric speed\cr
\AC_0([0,1];X) & Space of parametric curves with positive speed $\Leb{1}$-a.e. in $(0,1)$\cr
\AC^\infty_c([0,1];X)& Space of parametric curves with positive and constant speed\cr
\sfk & Embedding of $\bigl\{\gamma\in\AC([0,1];X):\ \ell(\gamma)>0\bigr\}$ into $\AC^\infty_c([0,1];X)$\cr
\Curvesnonpara{X} & Space of non-parametric and nonconstant curves, see Definition~\ref{def:nonparac}\cr
\sfi & Embedding of $\bigl\{\gamma\in\AC([0,1];X):\ \ell(\gamma)>0\bigr\}$ in $\Curvesnonpara X$\cr
\sfj & Embedding of $\Curvesnonpara{X}$ into $\AC^\infty_c([0,1];X)$\cr
\cr}

\begin{tikzcd}
 \Lipc{X}  \arrow[yshift=0.5ex]{r}{\pi_{\mathscr{C}}}
 & \Curvesnonpara{X} \arrow[yshift=-0.5ex]{l}{\sfj} \arrow{ddr}{\tilde{J}} \\
 \{\gamma\in\AC([0,1];X)\;:\;\ell(\gamma)>0\} \arrow{u}{\sfk} \arrow[swap]{ur}{\sfi}
\arrow[hookrightarrow]{d} \arrow{drr}{J} \\
 \rmC([0,1];X) \arrow{rr}{M} && \mathcal{M}_+(X)
\end{tikzcd}

\smallskip
\noindent {\bf Acknowledgement.} The first author acknowledges the support
of the ERC ADG GeMeThNES. The first and third author 
have been partially supported by PRIN10-11 grant from MIUR for the project Calculus of
Variations. All authors thank the reviewer, whose detailed comments led to an improvement of
the manuscript.

\part{{\large Duality between modulus and content}}

\section{Notation and preliminary notions}\label{sec:1}

In a topological Hausdorff space $(E,\tau)$,  we denote by $\mathscr{P}(E)$ the collection of all subsets of $E$, by
$\FF(E)$ (resp.~$\mathscr{K}(E)$) the collection of all closed
(resp.~compact) sets of $E$, by $\BorelSets{E}$ the $\sigma$-algebra of Borel sets
of $E$. We denote by $\rmC_b(E)$ the space of bounded  continuous functions on $(E,\tau)$, by $\mathcal{M}_+ (E) $, 
the set of $\sigma$-additive measures $\mu:\BorelSets{E}\to [0,\infty)$, by $\mathcal{P}(E)$ the subclass of
probability measures. 
For a set $F\subset E$ and $\mu\in\mathcal{M}_+(E)$ we shall respectively denote by $\chi_F:E\to \{0,1\}$ the characteristic 
function of $F$ and by  $\mu\res F$ the measure $\chi_F\mu$, if $F$ is $\mu$-measurable. For a Borel map $L:E\to F$ we shall
denote by $L_\sharp:\mathcal{M}_+(E)\to \mathcal{M}_+(F)$ the induced push-forward operator between Borel measures, namely
$$
L_\sharp\mu(B):=\mu\bigl(L^{-1}(B)\bigr)\qquad\forall \mu\in\mathcal{M}_+(E),\,\,B\in\BorelSets{F}.
$$
We shall denote by $\N=\{0,1,\ldots\}$ the natural numbers, 
by $\Leb{1}$ the Lebesgue measure on the real line.

\subsection{Polish spaces}

Recall that $(E,\tau)$ is said to be Polish if there exists a distance $\rho$ in $E$ which induces the topology
$\tau$ such that $(E,\rho)$ is complete and separable. Notice that the inclusion of $\mathcal{M}_+(E)$ in 
$(\rmC_b(E))^*$ may be strict, because we are not making compactness
or local compactness assumptions on $(E,\tau)$.  Nevertheless, if $(E,\tau)$ is Polish we can always endow $\mathcal{M}_+(E)$ with
a Polish topology $w\text{-}\rmC_b(E)$ whose convergent sequences are precisely the weakly convergent ones, i.e. sequences 
convergent in the duality with $\rmC_b(E)$. Obviously this Polish topology is unique. 
A possible choice,
which can be easily adapted from the corresponding 
Kantorovich-Rubinstein distance on $\Probabilities E$ 
(see e.g.~\cite[\S 8.3]{Bogachev} or 
\cite[Section~7.1]{Ambrosio-Gigli-Savare08})
is to consider the duality with bounded and
Lipschitz functions
\begin{align*}
  \rho_{K\kern-1pt R}(\mu,\nu):=\sup\Big\{\Big|\int_E f\,\d\mu-\int_E f\,\d\nu\Big|: \,&
  f\in \Lip_b(E),\ 
  \sup_E|f|\le 1,
    \\& |f(x)-f(y)|\le \rho(x,y)
  \quad \forall\, x,y\in E\Big\}.
\end{align*}

\subsection{Souslin, Lusin and analytic sets, Choquet theorem}

Denote by $\N^\infty$ the collection of all infinite sequences of natural numbers and by $\N^\infty_0$
the collection of all finite sequences $(n_0,\ldots,n_i)$, with $i\geq
0$ and $n_i$ natural numbers. Let $\AA\subset\mathscr{P}(E)$ containing the empty set (typical examples
are, in topological spaces $(E,\tau)$, the classes $\FF(E)$, $\mathscr K(E)$, $\BB(E)$). 
We call table of sets in $\AA$ a 
map $C$ associating to each
finite sequence $(n_0,\ldots,n_i)\in\N_0^\infty$ a set
$C_{(n_0,\ldots,n_i)}\in \AA$. 

 \begin{definition} [$\AA$-analytic sets] 
 A set $S\subset E$ is said to be 
  $\AA$-analytic
if there exists a table $C$ of sets in $\AA$ such that
$$S= \bigcup_{(n)\in\N^\infty } \bigcap_{i=0}^{\infty} C_{(n_0,\ldots, n_i)}. $$
  \end{definition}
  
 Recall that, in a topological space $(E,\tau)$, $\BB(E)$-analytic sets are \emph{universally measurable}
\cite[Theorem~1.10.5]{Bogachev}: this means that they are $\sigma$-measurable 
for any $\sigma\in{\mathcal M}_+(E)$. 
  
\begin{definition} [Souslin and Lusin sets]
  Let $(E,\tau)$ be an Hausdorff topological space. 
  $S\in\mathscr{P}(E)$ is said to be 
  a  Souslin (resp.~Lusin) set if it is the image of a Polish space
  under a continuous (resp.~continuous and injective) map.
\end{definition}

Even though the Souslin and Lusin properties for subsets of a topological space are intrinsic, i.e. they depend only
on the induced topology,  we will often use the diction ``$S$ Suslin subset of $E$'' and similar to emphasize the ambient space;
the Borel property, instead, is not intrinsic, since $S\in\BB(S)$ if we endow $S$ with the induced topology.
Besides the obvious stability with respect to transformations through 
continuous (resp. continuous and injective) maps,  the class of Souslin (resp. Lusin) sets enjoys nice properties, 
detailed below. 

\begin{proposition}\label{proplusin}
The following properties hold:
\begin{itemize}
\item[(i)]  In a Hausdorff topological space $(E,\tau)$, Souslin sets are $\FF(E)$-analytic;
\item[(ii)] if $(E,\tau)$ is a Souslin space (in particular 
if it is a Polish or a Lusin space),  the notions of Souslin and $\FF(E)$-analytic sets concide
and in this case Lusin sets are Borel and Borel sets are Souslin; 
\item[(iii)] if $E$, $F$ are Souslin spaces and $f:E\to F$ is a Borel injective map, then $f^{-1}$ is Borel;
\item[(iv)] if $E$, $F$ are Souslin spaces and $f:E\to F$ is a Borel map, then $f$ maps Souslin sets to Souslin sets.
\end{itemize}
\end{proposition}
\begin{proof} We quote \cite{Bogachev} for all these statements: (i) is proved in Theorem~6.6.8; in connection with (ii), the equivalence
between Souslin and $\FF(E)$-analytic sets is proved in Theorem~6.7.2, the fact that Borel sets are Souslin
in Corollary~6.6.7 and the fact that Lusin sets are Borel in Theorem~6.8.6; finally, (iii) and (iv) are proved in
Theorem~6.7.3.
\end{proof}

Since in Polish spaces $(E,\tau)$ we have at the same time tightness of finite Borel measures
and coincidence of Souslin and $\FF(E)$-analytic sets, the measurability of $\BB(E)$-analytic sets yields in particular that
\begin{equation}\label{eq:inner1}
\sigma(B)=\sup\left\{\sigma(K)\;:\; K\in\mathscr{K}(E),\,\,K\subset B\right\}
\quad\text{for all $B\subset E$ Souslin, $\sigma\in\mathcal{M}_+(E)$.}
\end{equation}

We will need a property analogous to \eqref{eq:inner1} for capacities
\cite{Dellacherie-Meyer}, whose definition is recalled below.

\begin{definition}[Capacity] A set function $\mathfrak{I} : \mathscr{P} (E) \to [0,\infty]$ is said to be a \emph{capacity} if:
\begin{itemize}
\item $\mathfrak{I} $ is nondecreasing and, whenever $(A_n)\subset\mathscr{P}(E)$ is nondecreasing, the following holds
$$ \lim_{n \to \infty} \mathfrak{I} (A_n) = \mathfrak{I}  \left( \bigcup_{n=0}^{\infty} A_n \right);$$
\item if $(K_n)\subset \mathscr{K}(E)$ is nonincreasing, the following holds:
$$ \lim_{n \to \infty} \mathfrak{I} (K_n) = \mathfrak{I}  \left( \bigcap_{n=0}^{\infty} K_n \right).$$
\end{itemize}
A set $B\subset E$ is said to be $\mathfrak{I}$-capacitable if $\mathfrak{I} (B)= \sup\limits_{K \in\mathscr{K}(E),\,K\subset B} \mathfrak{I} (K)$.
\end{definition}

\begin{theorem}[Choquet]
  \label{thm:Choquet}
  {\rm (\cite[Thm~28.III]{Dellacherie-Meyer})} Every $\KK(E)$-analytic set is capacitable.
\end{theorem}

\section{$(p,\mm)$-modulus ${\rm Mod}_{p,\smm}$} 

In this section $(X,\tau)$ is a topological 
space and $\mm$ is a fixed Borel and nonnegative reference measure, not necessarily finite or
$\sigma$-finite.

Given a power $p\in [1,\infty)$, we set
\begin{equation}\label{eq:callp}
\mathcal{L}^p_+ (X, \mm) := \left\{ f:X \to [0,\infty] \; : \; \text{$f$ Borel,    $\int_X f^p \,\d\mm < \infty$} \right\}.
\end{equation}
We stress that, unlike $L^p(X,\mm)$, this space is not quotiented under any equivalence relation; however we will
keep using the notation
$$ \| f\|_p := \left( \int_X |f|^p \, \d \mm \right)^{1/p}  $$
as a seminorm on $\Ldp$ and a norm in $L^p(X,\mm)$.

Given $\Sigma \subset \mathcal{M}_+$ we define (with the usual convention $\inf\emptyset=\infty$)
\begin{equation} \label{eqn:mod2}
{ \rm Mod}_{p,\smm} (\Sigma) := \inf \left\{ \int_X f^p \, \d \mm \, : \, f \in \Ldp,\,\, \int_X  f \, \d \mu \geq 1 \; \; \text{ for all } \mu \in \Sigma \right\},
\end{equation}
\begin{equation} \label{eqn:mod2c}
{ \rm Mod}_{p,\smm,c} (\Sigma) := \inf \left\{  \int_X f^p \, \d \mm \, : \, f \in \rmC_b(X,[0,\infty)),\,\,  \int_X  f \, \d \mu \geq 1 \; \; \text{ for all } \mu \in \Sigma \right\}.
\end{equation}
Equivalently, if $0<{\rm Mod}_{p,\smm}(\Sigma)\leq\infty$, we can say that ${ \rm Mod}_{p,\smm}(\Sigma)^{-1}$ is the least
number $\xi\in [0,\infty)$ such that the following is true
\begin{equation}\label{eqn:mod2dual} 
\left( \inf_{ \mu \in \Sigma} \int_X f \, \d \mu \right)^p
\leq  \xi \int_X f^p \, \d \mm \quad 
\text{for all $f \in \Ldp$},
\end{equation}
and similarly there is also an equivalent definition for ${ \rm Mod}_{p,\smm,c} (\Sigma)^{-1}$. \\
Notice that the infimum in \eqref{eqn:mod2c} is unchanged if we restrict the minimization to nonnegative functions $f\in \rmC_b(X)$. As a consequence,
since the finiteness of $\mm$ provides the inclusion of this class of functions in $\Ldp$, we get 
${ \rm Mod}_{p,\smm,c} (\Sigma) \geq { \rm Mod}_{p,\smm} (\Sigma)$ whenever $\mm$ is finite.
Also, whenever $\Sigma$ contains the null measure, we have ${\rm Mod}_{p,\smm,c} (\Sigma) \geq { \rm Mod}_{p,\smm} (\Sigma)=\infty$.

\begin{definition}[${\rm Mod}_{p,\smm}$-negligible sets]\label{dfn:modnull} A set $\Sigma \subset\mathcal{M}_+ (X)$ is said to be 
${\rm Mod}_{p,\smm}$-negligible if ${\rm Mod}_{p,\smm}(\Sigma)=0$. 
\end{definition}

A property $P$ on $\mathcal{M}_+ (X)$ is said to hold ${\rm Mod}_{p,\smm}$-a.e. if the set 
$$
\left\{\mu\in\mathcal{M}_+(X):\ \text{$P(\mu)$ fails}\right\}
$$
is ${\rm Mod}_{p,\smm}$-negligible. With this terminology, we can also write
\begin{equation} \label{eqn:mod2bis}
{ \rm Mod}_{p,\smm} (\Sigma) = \inf \left\{  \int_X f^p \, \d \mm \, : \, \int_X  f \, \d \mu \geq 1 \; \; \text{ for ${\rm Mod}_{p,\smm}$-a.e. } \mu \in \Sigma \right\}.
\end{equation}

We list now some classical properties that will be useful in the sequel, most them are well known and simple to prove, but
we provide complete proofs for the reader's convenience. 

\begin{proposition}\label{prop:prop} The set functions $A\subset\mathcal{M}_+(X)\mapsto\Md(A)$,
$A\subset\mathcal{M}_+(X)\mapsto\Mdc(A)$ satisfy the following properties:
\begin{itemize}
\item[(i)] both are monotone and their $1/p$-th power is subadditive;
\item[(ii)] if $g\in\Ldp$ then $ \int_X g \, \d \mu < \infty$ for
  $\Md$-almost every $\mu$; conversely, if 
  $\Md(A)=0$ then there exists $g\in \Ldp$ such that 
  $\int_X g\,\d\mu=\infty$ for every $\mu\in A$.
\item[(iii)] if $(f_n) \subset \Ldp $ converges in $L^p(X,\mm)$ seminorm to $f \in \Ldp$,  
there exists a subsequence $(f_{n(k)})$ such that
\begin{equation} \label{convqo}
\int_X f_{n(k)} \, \d \mu \rightarrow \int_X f \, \d \mu \qquad \Md \text{-a.e. in $\mathcal{M}_+(X)$};
\end{equation}
\item[(iv)] if $p>1$, for every $\Sigma\subset\mathcal{M}_+(X)$ with $\Md(\Sigma)<\infty$ there exists $f\in \Ldp$, unique up to $\mm$-negligible sets, 
such that $\int_Xf\,\d\mu\geq 1$ $\Md$-a.e. on $\Sigma$ and $\|f\|^p_p=\Md(\Sigma)$;
\item[(v)] if $p>1$ and $A_n$ are nondecreasing subsets of $\mathcal{M}_+(X)$ then $\Md (A_n) \uparrow \Md (\cup_n A_n)$;
\item[(vi)] if $K_n$ are nonincreasing compact subsets of
  $\mathcal{M}_+(X)$ then $\Mdc (K_n) \downarrow \Mdc (\cap_nK_n)$.
\item[(vii)] Let $A\subset \mathcal M_+(X)$,
  $F:A\to (0,\infty)$ be a Borel map, and
  $B=\big\{F(\mu)\mu:\ \mu\in A\big\}$. 
  If $\Md(A)=0$ then $\Md(B)=0$ as well.
\end{itemize}
\end{proposition}
\begin{proof} (i) Monotonicity is trivial. For the subadditivity, if we take $\int_X f\,\d\mu\geq 1$ on $A$ and $\int_Xg\,\d\mu\geq 1$ on $B$, 
then $\int_X (f+g)\,\d\mu \geq 1$ on $A \cup B$, hence
$\Md( A \cup B)^{1/p} \leq\|f+g\|_p\leq\|f\|_p+\|g\|_p$. Minimizing over $f$ and $g$ we get the subadditivity.

(ii) Let us consider the set where the property fails:
$$ \Sigma_g = \left\{ \mu \in \mathcal{M}_+ (X) \; : \; \int_X g \, \d \mu = \infty \right\}. $$
Then it is clear that $\Md(\Sigma_g)\leq\|g\|^p_p$ but $\Sigma_g = \Sigma_{\lambda g}$ for every $ \lambda >0$ 
and so we get that $\Sigma_g$ is $\Md$-negligible. Conversely, if $\Md(A)=0$ for every $n\in\N$ we can find $g_n\in \Ldp$
with
$\int_X g_n\,\d\mu\ge1$ for every $\mu\in A$ and 
$\int_X g_n^p\le 2^{-np}$. Thus $g:=\sum_n g_n$ satisfies the 
required properties.

(iii) Let $f_{n(k)}$ be a subsequence such that $\|f- f_{n(k)}\|_p\leq 2^{-k}$ so that if we set
$$ g (x) = \sum_{k =1}^{\infty} | f(x) - f_{n(k)} (x) |$$
we have that $g \in \Ldp$ and $\|g\|_p\leq 1$; in particular we have, for (ii) above, that $\int_X g \, \d \mu$ is finite for $\Md$-almost every $\mu$.
For those $\mu$ we get
$$ \sum_{k=1}^\infty \int_X |f - f_{n(k)} | \, \d \mu <\infty $$
and thus we get \eqref{convqo}.

(iv) Since we can use \eqref{eqn:mod2bis} to compute ${\rm Mod}_{p,\smm}(\Sigma)$, we obtain from (ii) and (iii)
that the class of admissible functions $f$ is a convex and closed subset of the Lebesgue space $L^p$.
Hence, uniqueness follows by the strict convexity of the $L^p$ norm.

(v) By the monotonicity, it is clear that $\Md (A_n)$ is an increasing sequence and that $\Md(\cup_n A_n) \geq \lim \Md(A_n) =:C$. 
If $C= \infty$ there is nothing to prove, otherwise, we need to show that $\Md (\cup_nA_n) \leq C$; let $(f_n) \subset \Ldp$ be a sequence of functions 
such that $ \int_X f_n\,\d\mu \geq 1$ on $A_n$ and $\|f_n\|^p_p \leq \Md (A_n) + \frac 1n$.
In particular we get that $\limsup_n \|f_n\|^p_p =C < \infty$ and so, possibly extracting a subsequence, we can assume that
$f_n$ weakly converge to some $f \in \Ldp$. By Mazur lemma we can find convex combinations
$$ \hat{f}_n  = \sum_{k=n}^{\infty} \lambda_{k,n} f_k $$
such that $\hat{f}_n$ converge strongly to $f$ in $L^p(X, \mm)$; furthermore we have that $\int_X f_k\,\d\mu \geq 1$ on $A_n$ if $k\geq n$ and so 
$$ \int_X\hat{f}_n\,\d\mu = \sum_{k=n}^{\infty} \lambda_{k,n} \int_Xf_k\,\d\mu\geq 1 \qquad \text{ on }A_n.$$
By (iii) in this proposition we obtain a subsequence $n(k)$ and a $\Md$-negligible set $\Sigma \subset \mathcal{M}_+(X)$ such that 
$\int_X\hat{f}_{n(k)}\,\d\mu\to\int_Xf\,\d\mu$ outside $\Sigma$; in particular $\int_X f\,\d\mu \geq 1$ on $\cup_nA_n \setminus \Sigma$.
Then, by the very definition of $\Md$-negligible set, for every $\ep>0$ we can find $g_{\ep} \in \Ldp$ such that $\|g_{\ep} \|^p_p \leq \ep$ and 
$ \int_Xg_\ep\,\d\mu \geq 1$ on $\Sigma$, so that we have $ \int_X(f+ g_{\ep}) \,\d\mu\geq 1$ on $\cup_n A_n$ and 
$$ \Md ( \cup_nA_n)^{1/p} \leq \| g_{\ep}+ f\|_p \leq  \|g_{\ep}\|_p+ \|f \|_p \leq \ep^{1/p} + \liminf \|f_n\|_p \leq \ep^{1/p} + C^{1/p}. $$
Letting $\ep \to 0$ and taking the $p$-th power the inequality $\Md(\cup_nA_n)\leq\sup_n\Md(A_n)$ follows.

(vi) Let $K=\cap_n K_n$. As before, by the monotonicity we get $\Mdc ( K) \leq \Mdc(K_n)$ and so calling $C$ the limit of $\Mdc(K_n)$ as $n$ goes to infinity, we only have to prove $\Mdc(K) \geq C$. First, we deal with the case $\Mdc(K)>0$: using the equivalent definition, let $\phi_{\ep}\in \rmC_b(X)$ be such that $\| \phi_{\ep} \|_p =1$ and
$$\inf_{\mu\in K} \int_X\phi_{\ep}\,\d\mu %=\min_K \Phi(\phi_{\ep}) 
\geq \frac 1{\Mdc (K)^{1/p}}- \ep. $$
By the compactness of $K$ and of $K_n$, it is clear that the infimum above is a minimum and that
$\min\limits_{K_n} \int_X\phi_{\ep}\,d\mu \to \min\limits_K \int_X\phi_{\ep}\,\d\mu$,  so that
$$ \frac 1{C^{1/p}} = \lim_{n \to \infty} \frac 1 { \Mdc (K_n)^{1/p}} \geq \lim_{n \to \infty} \min_{\mu\in K_n} \int_X\phi_{\ep}\,\d\mu \geq \frac 1{\Mdc(K)^{1/p} } - \ep. $$ 
The case $\Mdc(K)=0$ is the same, taking $\phi_M\in \rmC_b(X)$ such
that $\|\phi_M\|_p=1$ and $\int_X\phi_M\,d\mu\geq M$ on $K$ and then
letting $M \to \infty$.

(vii) Since $\Md(A)=0$, by (ii) we find $g\in \Ldp$ such
that $\int_X g\,\d\mu=\infty$ for every $\mu\in A$: this yields
$\int_X g\,\d\big(F(\mu)\mu\big)=\infty$ for every $\mu\in A$, 
showing that  $\Md(B)=0$.
\end{proof}

\begin{remark} \label{rem:saturated} {\rm In connection with Proposition~\ref{prop:prop}(iv), in general the constraint $\int_Xf\,d\mu\geq 1$ is not saturated
by the optimal $f$, namely the strict inequality can occur for a subset $\Sigma_0$ with positive $(p,\mm)$-modulus. For instance, if
$X=[0,1]$ and $\mm$ is the Lebesgue measure, then
$$
{\rm Mod}_{p,\smm}\bigl(\{\Leb{1}\res [0,\frac 12],\Leb{1}\res [\frac 12,1],\Leb{1}\res [0,1]\}\bigr)=2^p\quad\text{and}\quad f\equiv 2,
$$
but $\int_X f\,\d\mm=2$.
However, we will prove using the duality formula $\Md=C^p_{p,\smm}$ that one can always find a subset $\Sigma'\subset\Sigma$ (in the example
above $\Sigma\setminus\Sigma'=\{\Leb{1}\res [0,1]\}$) with the same
$(p,\mm)$-modulus satisfying $\int_Xf\,d\mu=1$ for all $\mu\in\Sigma'$, see the comment made after Corollary~\ref{cor:cnes}.

On the other hand, if the measures in $\Sigma$ are non-atomic, using just the definition of $p$-modulus, one can find instead a family $\Sigma'$ of \emph{smaller}
measures with the same modulus as $\Sigma$ on which the constraint is saturated: suffices to find, for any $\mu\in\Sigma$, a smaller
measure $\mu'$ (a subcurve, in the case of measures associated to curves) satisfying $\int_X f\,d\mu'=1$. In the previous example the two constructions lead to the same result, but the two procedures
are conceptually quite different.}
\end{remark}

Another important property is the tightness of $\Md$ in $\mathcal{M}_+(X)$: it will play a crucial role 
in the proof of Theorem \ref{tmain} to prove the inner regularity 
of $\Md$ for arbitrary Souslin sets.

\begin{lemma}[Tightness of $\Md$] \label{lem:tightness}
If $(X,\tau)$ is Polish and $\mm\in\mathcal{M}_+(X)$, 
for every $\ep>0$ there exists $E_{\ep}\subset\mathcal{M}_+(X)$ compact such that $\Md(E_\ep^c)\leq\ep$.
\end{lemma}
\begin{proof} Since $(X,\tau)$ is Polish, by Ulam theorem we can find a nondecreasing family of sets $K_n\in\mathscr{K}(X)$ such that
$$\mm(K_n^c)\to 0.$$
We claim the existence of $\delta_n \downarrow 0$ such that, defining 
$$E_k =\left\{ \mu \in \mathcal{M}_+(X)\; : \; \mu(X) \leq k \text{ and } \mu(K_n^c) \leq \delta_n \;  \; \forall n \geq k \right\},$$
then $E_k$ is compact and $\Md(E_k^c) \to 0$ as $k$ goes to infinity. First of all it is easy to see that the family
$\{E_k\}$ is compact by Prokhorov  theorem,  because it is clearly tight. 

To evaluate $\Md(E_k^c)$ we have to build some functions.
Let $m_n=\mm(K_n^c)$, assume with no loss of generality that $m_n>0$ for all $n$, set
$a_n=(\sqrt{m_n}+\sqrt{m_{n+1}})^{-1/p}$ and note that this latter sequence is nondecreasing and diverging to $+\infty$; let
us now define the functions
$$ f_k(x) : =\begin{cases} 0  & \text{ if }x \in K_k, \\ a_n & \text{ if }x \in K_{n+1} \setminus K_n \text{ and }n\geq k, \\ + \infty &\text{otherwise.} \end{cases}$$
Now we claim that if we put $\delta_n=a_n^{-1}$ in the definition of the $E_k$'s we will have $\Md(E_k^c) \to 0$: in fact, if $\mu \in E_k^c$ then we have either $ \mu(X) >k $ or $\mu(K_n^c) > \delta_n$ for some $n \geq k$. In either case the integral of the function $f_k+ \frac 1k$ with respect to $\mu$ is greater or equal to $1$:
\begin{itemize}
\item if $\mu(X)>k$ then
$$\int_X \left(f_k+ \frac 1k \right) \, \d \mu \geq \int_X \frac 1k \, \d \mu \geq 1;$$
\item if $\mu(K_n^c)> \delta_n$ for some $n\geq k$ we have that
$$ \int_X \left(f_k + \frac 1k \right) \, \d \mu \geq \int_{K_n^c} f_k \, \d \mu \geq \int_{K_n^c} a_n \, \d \mu > \delta_n a_n=1.$$
\end{itemize}
So we have that $\Md(E_k^c) \leq \|f_k+ \frac 1k\|^p_p \leq ( \|f_k\|_p + \| 1/k \|_p )^p$. But
$$\int_X f_k^p \,\d\mm =\sum_{n=k}^{\infty}\int_{K_{n+1} \setminus K_n} a_n^p \, \d \mm = 
\sum_{n=k}^{\infty}  \frac{m_{n} - m_{n+1}} { \sqrt{ m_{n} } + \sqrt{m_{n+1}}} =  \sum_{n=k}^{\infty} ( \sqrt{m_n} - \sqrt{m_{n+1}} ) = \sqrt{m_k},$$
and so we have $\Md(E_k^c) \leq \Bigl( (m_k)^{1/(2p)} + (\mm(X))^{1/p}/k \Bigr)^p \to 0$.
\end{proof}

\section{Plans with barycenter in $L^q(X,\mm)$ and $(p,\mm)$-capacity}

In this section $(X,\tau)$ is Polish and $\mm\in {\mathcal M}_+(X)$ is a fixed reference measure.
We will endow ${\mathcal M}_+(X)$ with the Polish structure making the maps $\mu\mapsto\int_X f\,d\mu$,
$f\in \rmC_b(X)$, continuous, as described in Section~\ref{sec:1}.

\begin{definition}[Plans with barycenter in $L^q(X,\mm)$]
Let $q\in (1,\infty]$, $p=q'$.
We say that a Borel probability measure $\eeta$ on $\mathcal{M}_+(X)$ is a plan with barycenter
in $L^q(X,\mm)$
if there exists $c\in [0,\infty)$ such that
\begin{equation}\label{eq:defboundcomp}
\iint_X f \, \d \mu \, \d \eeta(\mu)  \leq  c \| f \|_p\qquad \forall f \in \Ldp.
\end{equation}
If $\eeta$ is a plan with barycenter in $L^q(X,\mm)$, we call $c_q(\eeta)$ the minimal $c$ in \eqref{eq:defboundcomp}.
\end{definition}

Notice that $c_q(\eeta)=0$ iff $\eeta$ is the Dirac mass at the null measure in $\mathcal{M}_+(X)$. We also used implicitly
in \eqref{eq:defboundcomp} (and in the sequel it will be used without further mention) the fact that $\mu\mapsto\int_X f\,\d\mu$
is Borel whenever $f\in\Ldp$. The proof can be achieved by a standard monotone class argument.

An equivalent definition of the class plans with 
barycenter in $L^q(X,\mm)$, which explains also the terminology we adopted, 
is based on the requirement that the barycenter Borel measure
\begin{equation}\label{eq:defunderlinemu}
\underline{\mu}:=\int \mu\,\d\eeta(\mu)
\end{equation}
is absolutely continuous w.r.t. $\mm$ and with a density $\rho$ in $L^q(X,\mm)$. Moreover, 
\begin{equation}\label{eq:cppirhoq}
c_q(\eeta)=\|\rho\|_q.
\end{equation}
Indeed, choosing $f=\chi_A$ in \eqref{eq:defboundcomp} gives $\underline{\mu}(A)\leq (\mm(A))^{1/p}$,
hence the Radon-Nikodym theorem provides the representation $\underline{\mu}=\rho\mm$ for some $\rho\in L^1(X,\mm)$.
Then, \eqref{eq:defboundcomp} once more gives
$$
\int_X \rho f\,\d\mm\leq c\|f\|_p\qquad\forall f\in L^p(X,\mm)
$$
and the duality of Lebesgue spaces gives $\rho\in L^q(X,\mm)$ and $\|\rho\|_q\leq c$. Conversely, if $\underline{\mu}$
has a density in $L^q(X,\mm)$, we obtain by H\"older's inequality that  \eqref{eq:defboundcomp} holds with
$c=\|\rho\|_q$.

Obviously, \eqref{eq:defboundcomp} still holds with $c=c_q(\eeta)$ for all $f\in \rmC_b(X)$, not necessarily nonnegative, when
$\eeta$ is a plan with good barycenter in $L^q(X,\mm)$. Actually the next proposition shows that we need only to check
the inequality \eqref{eq:defboundcomp} for $f\in \rmC_b(X)$ nonnegative.

\begin{proposition}\label{prop:monotone} Let $\eeta$ be a probability measure on $\mathcal{M}_+(X)$ such that
\begin{equation}\label{eq:disug} \int\int_Xf\,\d\mu \, \d \eeta(\mu) \leq c \|f\|_p\qquad\text{for all $f\in \rmC_b(X)$ nonnegative}
\end{equation}
for some $c\geq 0$. Then \eqref{eq:disug} holds, with the same constant $c$, also for every $f \in \Ldp$.
\end{proposition}

\begin{proof} It suffices to remark that \eqref{eq:disug} gives
$$
\int_X f\,\d\underline{\mu}\leq c\|f\|_p\qquad\forall f\in \rmC_b(X),
$$
with $\underline{\mu}$ defined in \eqref{eq:defunderlinemu}. Again the duality of Lebesgue spaces provides $\rho\in L^q(X,\mm)$
with $\|\rho\|_q\leq c$ satisfying $\int_Xf\rho\,\d\mm=\int_Xf\,\d\underline{\mu}$ for all $f\in\rmC_b(X)$, hence $\underline{\mu}=\rho\mm$.
\end{proof}

There is a simple duality inequality, involving the minimization in \eqref{eqn:mod2} and a maximization among all $\eeta$'s with 
barycenter in $L^q(X,\mm)$. To see it, let's take $f \in \Ldp$ such that $\int f\,\d\mu\geq 1$ on 
$\Sigma\subset\mathcal{M}_+(X)$. 
Then, if $\Sigma$ is universally measurable we may take any plan $\eeta$ with barycenter in $L^q(X,\mm)$ to obtain
\begin{equation}\label{eq:sup1}
\eeta( \Sigma) \leq \int \int_X f\,\d\mu\, \d \eeta(\mu) \leq c_q(\eeta) \|f \|_p.
\end{equation}
In particular we have
\begin{equation}\label{eq:supp1}
\Md(\Sigma)=0\quad\Longrightarrow\quad \eeta(\Sigma)=0\qquad\text{for all $\eeta$ with barycenter in $L^q(X,\mm)$}.
\end{equation}
In addition, taking in \eqref{eq:sup1} the infimum over all the $f\in\Ldp$ such that $\int f\,\d\mu\geq 1$ on $\Sigma$
and, at the same time, the supremum with respect to all plans $\eeta$ with barycenter
in $L^q(X,\mm)$ and $c_q(\eeta)>0$,
we find 
\begin{equation}\label{eq:sup}
\sup_{c(\seeta)>0} \frac{\eeta( \Sigma)}{ c_q(\eeta)} \leq \Md( \Sigma)^{1/p}.
\end{equation}
The inequality \eqref{eq:sup} motivates the next definition.

\begin{definition}[$(p,\mm)$-content]
If $\Sigma\subset {\mathcal M}_+(X)$ is a universally measurable set we define
\begin{equation}\label{eq:defCp}
C_{p,\smm} ( \Sigma):= \sup_{c_q(\seeta)>0} \frac{\eeta( \Sigma)}{ c_q(\eeta)} .
\end{equation}
By convention, we set $C_{p,\smm}(\Sigma)=\infty$ if $0\in\Sigma$.
\end{definition}

A first important implication of \eqref{eq:sup} is that 
for any family $\mathcal F$ of plans $\eeta$ with 
barycenter in $L^q(X,\mm)$
\begin{equation}
  C:=\sup\left\{c_q(\eeta)\;:\;\eeta\in\mathcal
    F\right\}<\infty\quad\Longrightarrow\quad
  \cF\text{ is tight.}\label{eq:1}
\end{equation}
Indeed, $\eeta(E_{\ep^p}^c) \leq \ep c_q(\eeta) \leq C \ep$, where the $E_{\ep}\subset {\mathcal M}_+(X)$ are the compact sets provided by
Lemma~\ref{lem:tightness}. This allows to prove
existence of optimal $\eeta$'s in \eqref{eq:defCp}.

\begin{lemma} \label{lem:existsppi} Let $\Sigma\subset{\mathcal M}_+(X)$ be a universally measurable set such that
 $C_{p,\smm} (\Sigma) >0$ and $\sup_\Sigma\mu(X)<\infty$. Then there exists an optimal plan $\eeta$ with barycenter
in $L^q(X,\mm)$ in \eqref{eq:defCp}, and any optimal plan is concentrated on $\Sigma$. In particular
 $$ C_{p,\smm} ( \Sigma) = \frac{\eeta(\Sigma)}{c_q(\eeta)} = \frac 1{c_q(\eeta)}. $$
\end{lemma}

\begin{proof} First we claim that the supremum in \eqref{eq:sup} can be restricted to the plans with barycenter in $L^q(X,\mm)$
concentrated on $\Sigma$. Indeed, given any admissible $\eeta$ with $\eeta(\Sigma)>0$, defining $\eeta' = (\eeta(\Sigma))^{-1}\chi_{\Sigma}\eeta$ we
obtain another plan with barycenter in $L^q(X,\mm)$ satisfying $\eeta'(\Sigma)=1$ and
$$ 
\int \int_Xf\,\d\mu \, \d \eeta' (\mu)= \frac 1{\eeta(\Sigma) }  \int_{\Sigma} \int_X f\,\d\mu \, \d \eeta(\mu) 
\leq  \frac 1{\eeta(\Sigma) } \int \int_X f\,\d\mu \, \d \eeta (\mu)\leq \frac { c_q(\eeta)}{\eeta(\Sigma)} \|f \|_p 
$$
for all $f\in\Ldp$.
In particular the definition of $c_q(\eeta')$ gives
$$ {c_q(\eeta')} \leq \frac {c_q(\eeta)}{ \eeta (\Sigma)}, $$
and proves our claim. The same argument proves that $\eeta'=\eeta$ whenever $\eeta$ is a maximizer. Now we know that
$$ C_{p,\smm} ( \Sigma) = \sup_{ \seeta(\Sigma)=1} \frac 1 {c_q(\eeta)},$$
where the supremum is made over plans with barycenter in $L^q(X,\mm)$. We take a maximizing sequence $(\eeta_k)$; for this sequence we have 
that $c_q(\eeta_k)\leq C $, so that $(\eeta_k)$ is tight by \eqref{eq:1}. 
Assume with no loss of generality that $\eeta_k$ weakly converges to some $\eeta$, that is 
clearly a probability measure in ${\mathcal M}_+(X)$. To see that $\eeta$ is a plan with barycenter in $L^q(X,\mm)$ and that $c_q(\eeta)$ is optimal, 
we notice that the continuity and 
boundedness of $\mu\mapsto\int_X f\,\d\mu$ in bounded sets of ${\mathcal M}_+(X)$  for $f\in \rmC_b(X)$ gives
$$ \int \int_X f\,\d\mu \, \d \eeta (\mu)= \lim_{k \to \infty} \int \int_X f\,\d\mu \, \d \eeta_k(\mu) \leq \lim_{k \to \infty} c_q(\eeta_k) \|f\|_p,$$ 
so that
$$\int \int_X f\,\d\mu \, \d \eeta(\mu) \leq \frac 1{ C_{p,\smm} (\Sigma) } \|f \|_p \qquad \forall f \in \rmC_b(X).$$
The thesis follows from Proposition~\ref{prop:monotone}.
\end{proof}

\section{Equivalence between $C_{p,\smm}$ and $\Md$ }

In the previous two sections, under the standing assumptions 
$(X,\tau)$ Hausdorff topological space (Polish in the case of $C_{p,\mm}$), 
$\mu\in{\mathcal M}_+(X)$ and
$p\in [1,\infty)$, we introduced a $p$-modulus $\Md$ and a $p$-content
$C_{p,\mm}$, proving the direct inequalities (see \eqref{eq:sup})
$$C^p_{p,\smm} \leq \Md \leq \Mdc\qquad\text{on Souslin subsets of ${\mathcal M}_+(X)$.}$$

Under the same assumptions on $(X,\tau)$ and $\mm\in {\mathcal M}_+(X)$, our goal in this section is the following result:

\begin{theorem}\label{tmain} 
  Let $(X,\tau)$ be a Polish topological space and $p>1$. Then
  $\Md$ is a Choquet capacity in $\mathcal{M}_+(X)$, 
every Souslin set $\Sigma\subset{\mathcal M}_+(X)$ is
capacitable and satisfies
$\Md(\Sigma)^{1/p} = C_{p,\smm}(\Sigma)$. 
If moreover $\Sigma$ is also compact we have 
$\Md(\Sigma)=\Mdc(\Sigma)$.
\end{theorem}

\begin{proof} We split the proof in two steps:
\begin{itemize}
 \item first, prove that $\Mdc (\Sigma)^{1/p} \leq C_{p,\smm} (\Sigma)$ if $\Sigma$ is compact, so that
  in particular $\Md^{1/p} = C_{p,\smm} $ on compact sets;
 \item then, prove that $\Md$ and $C_{p,\smm}$ are inner regular, and deduce that $\Md^{1/p}=C_{p,\smm}$ on Souslin sets.
\end{itemize}
The two steps together yield $\Md=\Mdc$ on compact sets, hence we can use Proposition~\ref{prop:prop}(v,vi) to obtain that
 $\Md$ is a Choquet capacity
in $\mathcal{M}_+(X)$.

\noindent
{\bf Step 1.} Assume that $\Sigma\subset\mathcal{M}_+(X)$ is compact. In particular $\sup_\Sigma\mu(X)$ is finite and so we have that the
linear map $\Phi : \rmC_b(X) \to \rmC( \Sigma)=\rmC_b(\Sigma)$ given by
$$
f\mapsto \Phi_f(\mu):=\int_X f\,\d\mu
$$
is a bounded linear operator. 

If $\Sigma$ contains the null measure there is nothing to prove, because $\Mdc(\Sigma)=\infty$ by definition
and $C_{p,\mm}(\Sigma)=\infty$ by convention. If not, by compactness, we obtain that $\varepsilon:=\inf_\Sigma\mu(X)>0$, so that 
taking $f\equiv\varepsilon^{-1}$ in \eqref{eqn:mod2c} we obtain $\Mdc(\Sigma)<\infty$.
We can also assume that $\Mdc(\Sigma)>0$, otherwise there is nothing to prove.

Our first step is the construction of a plan $\eeta$ with barycenter in $L^q(X,\mm)$
concentrated on $\Sigma$. By the equivalent definition analogous to \eqref{eqn:mod2dual} for
$\Mdc$, the constant $\xi= \Mdc(\Sigma)^{-1/p}$ satisfies
\begin{equation}\label{eq:xi}\inf_{\mu\in\Sigma} \Phi_f(\mu) \leq \xi \| f \|_p\qquad\forall f\in \rmC(X).\end{equation}
Denoting by $v=v(\mu)$ a generic element of $\rmC(\Sigma)$, we will now consider two functions on $\rmC(\Sigma)$:
\begin{eqnarray*} & F_1 (v) &=\inf \left\{ \| f \|_p \;:\; f \in \rmC_b(X), \;\; \Phi_f \geq v\,\,\text{on $\Sigma$} \right\}\\
&F_2 (v) &= \min\left\{ v ( \mu ) \;:\;  \mu \in \Sigma \right\}. \end{eqnarray*}
The following properties are immediate to check, using the linearity of $f\mapsto\Phi_f$ for the first one
and \eqref{eq:xi} for the third one:
\begin{itemize}
\item $F_1$ is convex;
\item $F_2$ is continuous and concave;
\item $F_2 \leq \xi \cdot F_1$.
\end{itemize}
With these properties, standard Banach theory gives us a continuous linear functional $L \in (\rmC(\Sigma))^*$ such that   
\begin{equation}\label{eq:Hahn-Banach}
F_2(v)  \leq L (v) \leq \xi \cdot F_1 (v) \qquad \forall v \in \rmC(\Sigma). 
\end{equation}
For the reader's convenience we detail the argument: first we apply the geometric form of the Hahn-Banach theorem in the space $\rmC(\Sigma) \times \R$ to the convex sets $A = \{ F_2(v) > t \}$ 
and $B= \{ F_1 (v) \leq t/\xi \}$, where the former is also open, to obtain a continuous linear functional $G$ in $\rmC(\Sigma)\times\R$ such that
$$
G(v,t)<G(w,s)\qquad\text{whenever $F_2(v)>t$, $F_1(w)\leq s/\xi$.}
$$
Representing $G(v,t)$ as $H(v)+\beta t$ for some $H\in (\rmC(\Sigma))^*$ and $\beta\in\R$, the inequality reads
$$
H(v)+\beta t<H(w)+\beta s\qquad\text{whenever $F_2(v)>t$, $F_1(w)\leq s/\xi$.}
$$
Since $F_1$ and $F_2$ are real-valued, $\beta>0$; we immediately get $F_2\leq (\gamma-H)/\beta\leq\xi F_1$, with
$\gamma:=\sup H(v)+\beta F_2(v)$. On the other hand, $F_1(0)=F_2(0)=0$ implies $\gamma=0$, so that we can take
$L=-H/\beta$ in \eqref{eq:Hahn-Banach}.

In particular from \eqref{eq:Hahn-Banach} we get that if $v \geq 0$ then $L(v) \geq F_2 (v) \geq 0$ and so, since $\Sigma$ is compact, we can apply Riesz theorem
to obtain a nonnegative measure $\eeta$ in $\Sigma$ representing $L$:
$$L ( v) = \int_{\Sigma} v ( \mu ) \, \d \eeta\qquad\forall v\in \rmC(\Sigma). $$
Furthermore this measure can't be null since (here $\mathbbm{1}$ is the function identically equal to $1$).
$$ \eeta (\Sigma) = L(  \mathbbm{1} ) \geq F_2(  \mathbbm{1}) =1,$$
and so $\eeta (\Sigma) \geq 1$. Now we claim that $\eeta$ is a plan with barycenter in $L^q(X,\mm)$; first we prove that $\eeta(\Sigma)  \leq 1$, so that $\eeta$ will be a probability measure. In fact, we know $F_2(v) \eeta ( \Sigma) \leq L (v) $ because $v \geq F_2(v)$ on $\Sigma$, and then 
$$  F_2(v) \eeta (\Sigma)  \leq \xi F_1(v). $$
In particular, inserting in this inequality $v= \Phi_\phi$ with $\phi \in \rmC_b(X)$, we obtain
$$  \inf_{\Sigma} \Phi_\phi \leq \frac {\xi}{\eeta (\Sigma) } \|\phi\|_p $$
and so $ \Mdc (\Sigma) \geq (\eeta(\Sigma) / \xi)^p = \eeta(\Sigma)^p \Mdc(\Sigma)$, which implies $\eeta(\Sigma)\leq 1$. Now we have that
\begin{equation}\label{eqn:mono}\int_{\Sigma} \left(  \int_{X} f \, \d \mu \right) \, \d \eeta = L (\Phi_f) \leq 
\xi \cdot F_1 ( \Phi_f ) \leq \xi \cdot \|f \|_p \qquad \forall f \in \rmC_b(X)
 \end{equation}
and so, by Proposition~\ref{prop:monotone}, this inequality is true for every $f\in \Ldp$, showing that $\eeta$ is a plan
with barycenter in $L^q(X,\mm)$; as a byproduct we gain also that $c_q(\eeta) \leq \xi$ that gives us, 
that $C_{p,\smm} ( \Sigma) \geq \Mdc( \Sigma)^{1/p}$, thus obtaining that
$$ C_{p,\smm}(\Sigma)=\Md(\Sigma)^{1/p} = \Mdc(\Sigma)^{1/p}.$$

\noindent {\bf Step 2.}
Now we will prove that $\Md$ and $C_{p,\smm}$ are both inner regular, namely their value on Souslin sets is the supremum of their
value on compact subsets. Inner regularity and equality on compact sets yield $C_{p,\smm}(B) =\Md(B)^{1/p}$ on every Souslin subset $B$
of $\mathcal{M}_+ (X)$.

\smallskip\noindent
\textbf{$\Md$ is inner regular.} Proposition~\ref{prop:prop}(v,vi) and the fact that $\Mdc = \Md$ if the set is compact, give us that $\Md$ is a capacity. 
For any set $L\subset{\mathcal M}_+(X)$ we have $\Md (L)=\sup_{\ep} \Md(L \cap E_{\ep})$, where
$E_\ep$ are the compact sets given by Lemma~\ref{lem:tightness}.
Therefore, suffices to show inner regularity for a Souslin set $B$ contained in $E_\ep$ for some $\ep$. Since $E_\ep$ is compact,
$B$ is a Souslin-compact set and from Choquet Theorem~\ref{thm:Choquet} it follows that for every $\delta>0$ 
there is a compact set $K \subset B$ such that $\Md(K) \geq \Md(B) - \delta$.

\smallskip\noindent
\textbf{$C_{p,\smm}$ is inner regular.} Since Souslin sets are universally measurable and ${\mathcal M}_+(X)$ is Polish, 
we can apply \eqref{eq:inner1} to any Souslin set $B$ with $\sigma=\eeta$ to get
 $$ \sup_{K \subset B} C_{p,\smm} (K) =  \sup_{K \subset B} \sup_{c_q(\seeta)>0} \frac{\eeta (K)}{c_q(\eeta)} =  
 \sup_{c_q(\seeta)>0 } \sup_{K \subset B} \frac{\eeta (K)}{c_q(\eeta)}  = \sup_{c_q(\seeta)>0} \frac{\eeta (B)}{c_q(\eeta)} = C_{p,\smm} (B).$$
 \end{proof}

The duality formula and the existence of maximizers and minimizers provide the following result.

\begin{corollary} [Necessary and sufficient optimality conditions]\label{cor:cnes}
Let $p>1$, let $\Sigma \subset\mathcal{M}_+ (X)$ be  a Souslin set such that $\Md (\Sigma)>0$ and $\sup_\Sigma\mu(X)$ is finite. Then:
\begin{itemize}
\item[(a)] there exists $f \in \Ldp$, unique up to $\mm$-negligible sets,
 such that $\int_X f\,\d\mu \geq1$ for $\Md$-a.e. $\mu\in\Sigma$ and such that $\| f \|^p_p= \Md (\Sigma)$;
\item[(b)] there exists a plan $\eeta$ with barycenter in $L^q(X,\mm)$ concentrated on $\Sigma$ such that $C_{p,\smm}(\Sigma) = 1/c_q(\eeta)$;
\item[(c)] for the function $f$ in (a) and any $\eeta$ in (b) there
  holds
\begin{equation}\label{eq:duetesi}
\int_X f\,\d\mu=1  \,\,\,\text{for $\eeta$-a.e. $\mu$} \qquad\text{and}\qquad
\int_X \mu\, \d \eeta (\mu) = \frac{f^{p-1}}{\|f\|_p^p}  \mm.
\end{equation}
\end{itemize}
Finally, if $f\in\Ldp$ is optimal in \eqref{eqn:mod2}, then any plan $\eeta$ with barycenter in $L^q(X,\mm)$
concentrated on $\Sigma$ such that $c_q(\eeta)=\|f\|_p^{-1}$ is
optimal in \eqref{eq:defCp}. Conversely, if $\eeta$ is optimal
in \eqref{eq:defCp}, $f\in\Ldp$ and $\int_Xf\,\d\mu=1$ for
$\eeta$-a.e. $\mu$ then $f$ is optimal in \eqref{eqn:mod2}. 
\end{corollary}
\begin{proof}
The existence of $f$ follows by Proposition~\ref{prop:prop}(iv). The existence of a maximizer $\eeta$ in the duality formula,
concentrated on $\Sigma$ and satisfying $C_{p,\mm}(\Sigma)=1/c_q(\eeta)$  
follows by Lemma~\ref{lem:existsppi}. Since \eqref{eq:supp1} gives $\int_X f\,\d\mu \geq1$ for $\eeta$-a.e. $\mu\in\Sigma$ we can still
derive the inequality \eqref{eq:sup1} and obtain from Theorem~\ref{tmain} that all inequalities are equalities.
Hence, $\int_X f\,\d\mu=1$ for $\eeta$-a.e. $\mu\in\mathcal{M}_+(X)$. Finally, setting $\underline{\mu}:=\int \mu\,\d\eeta(\mu)$,
from \eqref{eq:cppirhoq} we get $\underline{\mu}=g\mm$ with $\|g\|_q=c_q(\eeta)$. This, in combination with
$$
\int_X fg\,\d\mm=\int \int_X f\,d\mu\,\d\eeta(\mu)=c_q(\eeta)\|f\|_p=\|g\|_q\|f\|_p,
$$
gives $g=f^{p-1}/\|f\|_p^p$.

Finally, the last statements follow directly from \eqref{eq:sup1} and Theorem~\ref{tmain}.
\end{proof}

In particular, choosing $\eeta$ as in (b) and defining
$$
\Sigma':=\left\{\mu\in\mathcal{M}_+(X):\ \int_X f\,d\mu=1\right\},
$$
since $\eeta(\Sigma)=\eeta(\Sigma')$ 
we obtain a subfamily with the same $p$-modulus on which the constraint is saturated.

\part{{\large Modulus of families of curves and weak gradients}}

\section{Absolutely continuous curves}

If $(X,\sfd)$ is a metric space and $\J\subset\R$ is an interval, we denote by $\rmC(\J;X)$ the class of
continuous maps (often called parametric curves) from $\J$ to $X$. We will use the notation $\gamma_t$ for the value of the
map at time $t$ and $\e_t:\rmC(\J;X)\to X$ for the evaluation map at time $t$; occasionally, in order to avoid double
subscripts, we will also use the notation
$\gamma(t)$. The subclass $\AC(\J;X)$ is defined by
the property
$$
\sfd(\gamma_s,\gamma_t)\leq \int_s^tg(r)\,\d r
\qquad s,\,t\in \J,\,\,s\leq t
$$
for some (nonnegative) $g\in L^1(\J)$.
The least, up to $\Leb{1}$-negligible sets, function $g$ with this property is the so-called metric derivative
(or metric speed)
$$
|\dot\gamma_t|:=\lim_{h\to 0}\frac{\sfd(\gamma_{t+h},\gamma_t)}{|h|},
$$
see \cite{Ambrosio-Tilli}.
The classes $\AC^p(\J;X)$, $1\leq p\leq\infty$ are defined analogously,
requiring that $|\dot\gamma|\in L^p(\J)$.
The $p$-energy of a curve is then defined as 
\begin{equation}
  \label{eq:4}
  \cE_p(\gamma):=
  \begin{cases}
    \int_\J |\dot\gamma_t|^p\,\d t&\text{if }\gamma\in \AC^p(\J;X),\\
    +\infty&\text{otherwise,}
  \end{cases}
\end{equation}
and $\cE_1(\gamma)=\ell(\gamma)$, the length of $\gamma$, when $p=1$.
Notice that $\AC^1=\AC$ and that $\AC^\infty(\J;X)$ coincides with the class of $\sfd$-Lipschitz functions.

If $(X,\sfd)$ is complete the interval $\J$ can be taken closed with no loss of generality, because
absolutely continuous functions extend continuously to the closure of the interval. In addition,
if $(X,\sfd)$ is complete and separable then $\rmC(\J;X)$ is a Polish space, and $\AC^p(\J;X)$,
$1\leq p\leq\infty$ are Borel subsets of $\rmC(\J;X)$ (see for instance \cite{Ambrosio-Gigli-Savare11}).
We will use the short notation $\mathcal{M}_+(\AC^p(\J;X))$ to denote finite Borel measures in
$\rmC(\J;X)$ concentrated on $\AC^p(\J;X)$.

\subsection{Reparameterization}

We collect in the next proposition a few properties which are well-known in a smooth setting, but still
valid in general metric spaces. We introduce the notation
\begin{equation}
  \label{eq:14bis}
  \Lipc X:=\Big\{\sigma\in \AC^\infty([0,1];X):|\dot
  \sigma|=\ell(\sigma)>0
  \quad\text{$\Leb 1$-a.e. on $(0,1)$}\Big\}
\end{equation}
for the subset of $\AC([0,1];X)$ consisting of all nonconstant curves with constant speed.
It is easy to check that $\Lipc X$ is a Borel subset of 
$\rmC([0,1];X)$, since it can also be characterized by
\begin{equation}
  \label{eq:15}
  \gamma\in \Lipc X\quad\Longleftrightarrow\quad
  0<\Lip(\gamma)\le \ell(\gamma),
\end{equation}
and the maps $\gamma\mapsto \Lip(\gamma)$ and 
$\gamma\mapsto \ell(\gamma)$ are lower semicontinuous.

\begin{proposition} [Constant speed reparameterization]\label{prop:tilli} 
For any $\gamma\in \AC([0,1];X)$ with $\ell(\gamma)>0$, setting
  \begin{equation}
  \sfs(t):=\frac{1}{\ell(\gamma)}\int_0^t|\dot\gamma_r|\,\d r,
  \end{equation} 
  there exists a unique $\eta\in\AC^\infty_c([0,1];X)$  
  such that $\gamma=\eta\circ\sfs$.  Furthermore, $\eta=\gamma\circ\sfs^{-1}$ where
  $\sfs^{-1}$ is any right inverse of $\sfs$. We shall denote by
  \begin{equation}
  \sfk:\bigl\{\gamma\in\AC([0,1];X): \ell(\gamma)>0\bigr\}\to \AC^\infty_c([0,1];X)\qquad \gamma\mapsto\eta=\gamma\circ\sfs^{-1}
  \end{equation}
  the corresponding map.
\end{proposition}
\begin{proof} We prove existence only, the proof of uniqueness being analogous. 
Les us now define a right inverse, denoted by $\sfs^{-1}$, of $\sfs$ (i.e. $\sfs\circ \sfs^{-1}$ is equal to the identity): 
we define in the obvious way $\sfs^{-1}$ at points $y\in [0,1]$ 
such that $\sfs^{-1}(y)$ is a singleton; since, by construction, 
$\gamma$ is constant in all (maximal) intervals $[c,d]$ where $\sfs$ is constant, at points $y$ such that $\{y\}=\sfs([c,d])$ 
we define $\sfs^{-1}(y)$ by choosing any element of $[c,d]$, so that $\gamma\circ \sfs^{-1}\circ\sfs=\gamma$ (even though
it could be that $\sfs^{-1}\circ\sfs$ is not the identity). Therefore, 
if we define $\eta=\gamma\circ \sfs^{-1}$, we obtain that $\gamma=\eta\circ \sfs$ and that $\eta$ is independent of the chosen
right inverse. 

In order to prove that $\eta\in\AC^\infty_c([0,1];X)$ we define $\ell_k:=\ell(\gamma)+1/k$ and we
approximate uniformly in $[0,1]$ the map $\sfs$ by the maps 
$\sfs_k(t):=\ell_k^{-1}\int_0^t(k^{-1}+|\dot\gamma_r|)\, \d r$, whose inverses
$\sfs_k^{-1}:[0,1]\to\J$ are Lipschitz. By Helly's theorem and passing to the limit as $k\to\infty$ in $\sfs_k\circ\sfs_k^{-1}(y)=y$, 
we can assume that a subsequence $\sfs_{k(p)}^{-1}$ pointwise converges to
a right inverse $\sfs^{-1}$ as $p\to\infty$; the curves $\eta^p:=\gamma\circ\sfs_{k(p)}^{-1}$ are absolutely continuous, pointwise
converge to $\eta:=\gamma\circ\sfs^{-1}$ and
$$
|\eta^p(t)'|= \frac{|\gamma'(\sfs_{k(p)}^{-1}(t))|}{\sfs_{k(p)}'(\sfs_{k(p)}^{-1}(t))}\leq\ell_{k(p)}
\qquad\text{for $\Leb{1}$-a.e. in $t\in (0,1)$.}
$$
It follows that $\eta$ is absolutely continuous and that $|\dot\eta|\leq\ell(\gamma)$ $\Leb{1}$-a.e. in $(0,1)$. If the strict
inequality occurs in a set of positive Lebesgue measure, the inequality $\ell(\eta)<\ell(\gamma)$ provides a contradiction.
\end{proof}
 
\subsection{Equivalence relation in $\AC([0,1];X)$}

We can identify curves
$\gamma\in \AC([0,1],X)$, $\tilde\gamma\in\AC([0,1];X)$ if there exists $\varphi:[0,1]\to [0,1]$ increasing with $\varphi\in \AC([0,1];[0,1])$,
$\varphi^{-1}\in \AC([0,1];[0,1])$ such that
$\gamma=\tilde\gamma\circ\varphi$. In this case we write $\gamma\sim \tilde\gamma$.
Thanks to the following
lemma, the absolute continuity of $\varphi^{-1}$ is equivalent to $\varphi'>0$ $\Leb{1}$-a.e. in $(0,1)$.

\begin{lemma}[Absolute continuity criterion]\label{simplelemma}
Let $\J$, $\tilde{\J}$ be compact intervals in $\R$ and let $\varphi:\J\to\tilde{\J}$ be an absolutely continuous
homeomorphism with $\varphi'>0$ $\Leb{1}$-a.e. in $\J$. Then $\varphi^{-1}:\tilde{\J}\to \J$ is absolutely continuous.
\end{lemma}
\begin{proof} Let $\psi=\varphi^{-1}$; it is a continuous function of bounded variation whose distributional
derivative we shall denote by $\mu$. Since $\mu([a,b])=\psi(b)-\psi(a)$ for all $0\leq a\leq b\leq 1$, we need to show 
that $\mu\ll\Leb{1}$. It is a general property of $BV$
functions  (see for instance \cite[Proposition~3.92]{AFP}) that $\mu(\psi^{-1}(B))=0$ for all Borel and $\Leb{1}$-negligible
sets $B\subset\R$. Choosing $B=\psi(E)$, where $E$ is a $\Leb{1}$-negligible set where the singular
part $\mu^s$ of $\mu$ is concentrated, the area formula gives
$$
\int_B\varphi'(s)\,\d s=\Leb{1}(E)=0,
$$
so that the positivity of $\varphi'$ gives $\Leb{1}(B)=0$. It follows that $\mu^s=0$.
\end{proof}

\begin{definition} [The map $J$] For any $\gamma\in\AC([0,1];X)$ we denote by $J\gamma\in\cM_+(X)$ the push forward under $\gamma$
of the measure $|\dot\gamma|\Leb{1}\res [0,1]$, namely
\begin{equation}\label{eq:5}
\int_X g\,\d J\gamma=\int_0^1 g(\gamma_t)|\dot\gamma_t|\,\d t\qquad\text{for all $g:X\to [0,\infty]$ Borel.}
\end{equation}
{In particular we have that $ J\gamma= J\eta$ whenever $\gamma \sim \eta$ and that $J\gamma=J \sfk \gamma$.}
\end{definition}

Although this will not play a role in the sequel, for completeness we provide an intrinsic description of the measure $J\gamma$.
We denote by $\Haus{1}$ the 1-dimensional Hausdorff
measure of a subset $B$ of $X$, namely $\Haus{1}(B)=\lim_{\delta\downarrow 0}\Haus{1}_\delta(B)$, where
$$
\Haus{1}_\delta(B):=\inf\left\{\sum_{i=0}^\infty{\rm diam}(B_i)\;:\; B\subset\bigcup_{i=0}^\infty B_i,\,\,\,{\rm diam}(B_i)<\delta\right\}
$$
(with the convention ${\rm diam}(\emptyset)=0$).
\begin{proposition}[Area formula]\label{prop:tillibis}
If $\gamma\in \AC([0,1];X)$, then for all $g:X\to [0,\infty]$ Borel the area formula holds:
\begin{equation}\label{eq:area}
\int_0^1 g(\gamma_t)|\dot\gamma_t|\,\d t=\int_X g(x)N(\gamma,x)\,\d\Haus{1}(x),
\end{equation}
where $N(\gamma,x):={\rm card}(\gamma^{-1}(x))$ is the multiplicity
function of $\gamma$. Equivalently, 
\begin{equation}
  \label{eq:555}
  J\gamma=N(\gamma,\cdot)\Haus{1} .
\end{equation}
\end{proposition}
\begin{proof}
For an elementary proof of \eqref{eq:area}, see for instance \cite[Theorem~3.4.6]{Ambrosio-Tilli}. %The verification
%of the independence of $N(\gamma,\cdot)\Haus{1}$ in the equivalence class is a simple consequence of \eqref{eq:area}. 
\end{proof}

\subsection{Non-parametric curves}

We can now introduce the class of non-parametric curves; notice that
we are conventionally excluding from this class the constant curves. We introduce the notation
$$
\AC_0([0,1];X):=\left\{\gamma\in \AC([0,1];X):\ \text{$|\dot\gamma|>0$ $\Leb{1}$-a.e. on $(0,1)$}\right\}.
$$
It is not difficult to show that $\AC_0([0,1];X)$ is a Borel subset of $\rmC([0,1];X)$.
In addition, Lemma~\ref{simplelemma} shows that for any $\gamma\in\AC_0([0,1];X)$ the curve
$\sfk\gamma\in\AC^\infty_c([0,1];X)$ is equivalent to $\gamma$.

\begin{definition}[The class $\Curvesnonpara{X}$ of non-parametric curves]\label{def:nonparac}
The class $\Curvesnonpara{X}$ is defined as
\begin{equation}\label{eq:curvequotient}
  \Curvesnonpara{X}:=\AC_0([0,1];X)\hbox{$/\hskip -3pt \sim$} \,,
\end{equation}
endowed with the quotient topology $\tau_{{\mathscr C}}$ and the canonical projection
$\pi_{\Curvesnonpara{X}}$.
\end{definition}

We shall denote the typical element of $\Curvesnonpara{X}$ either by $\pgamma$ or by $[\gamma]$, to mark a distinction with
the notation used for parametric curves. We will use the notation $\pgamma_{{\rm ini}}$ and $\pgamma_{{\rm fin}}$ the initial
and final point of the curve $\pgamma\in\Curvesnonpara{X}$, respectively. 

\begin{definition}[Canonical maps]\label{def:canonical}
We denote: 
\begin{itemize}
\item[(a)] by
${\sf i}:=\pi_{{\mathscr C}}\circ\sfk: \bigl\{\gamma\in\AC([0,1];X):\ \ell(\gamma)>0\bigr\}\to\Curvesnonpara{X}$ the
projection provided by Proposition~\ref{prop:tilli}, which coincides with the canonical projection $\pi_{\Curvesnonpara{X}}$ on the quotient
when restricted to $\AC_0([0,1];X)$;
\item[(b)] by $\sfj:=\sfk\circ\pi_{{\mathscr C}}^{-1}:\Curvesnonpara X\to \Lipc X$ the 
canonical representation of a non-parametric curve by a
parameterization in $[0,1]$ with constant velocity.
\item[(c)] by $\tilde J:\Curvesnonpara{X}\to {\cM}_+(X)\setminus\{0\}$
the quotient of the map $J$ in \eqref{eq:5}, defined by
\begin{equation}\label{eq:defJ}
\tilde J[\gamma]:= J\gamma.
\end{equation}
\end{itemize}
\end{definition}

\begin{lemma}[Measurable structure of $\Curvesnonpara{X}$]
  \label{le:Lusin} If $(X,\sfd)$ is complete and separable, 
  the space
  $(\Curvesnonpara{X},\tau_{\mathscr C})$ is a Lusin Hausdorff space and the 
  restriction of the map $\sfi$ to $\Lipc X$ is a Borel isomorphism.
  In particular, a collection of curves $\Gamma\subset \Curvesnonpara X$
  is Borel if and only if $\sfj(\Gamma)$ is Borel in $\rmC([0,1];X)$. Analogously,  
  $\Gamma\subset \Curvesnonpara X$
  is Souslin if and only if $\sfj(\Gamma)$ is Souslin in $\rmC([0,1];X)$.
\end{lemma}
\begin{proof}
  Let us first show that $(\Curvesnonpara X,\tau_{\mathscr C})$ 
  is Hausdorff. We argue by contradiction and we suppose that 
  there exist curves $\sfi(\sigma_i)
  \in \Curvesnonpara X$ with $\sigma_i\in \Lipc X$, $i=1,\,2$, 
  and a sequence of parameterizations $\sfs^n_i\in \AC([0,1];[0,1])$ 
  with $(\sfs^n_i)'>0$ $\Leb 1$-a.e.~in $(0,1)$, 
  such that 
  \begin{displaymath}
    \lim_{n\to\infty}\sup_{t\in [0,1]} \sfd(\sigma_1(\sfs^n_1(t)),
    \sigma_2(\sfs^n_2(t)))=0.
  \end{displaymath}
  Denoting by $\sfr_1^n(t):=\sfs^n_1\circ (\sfs_2^n)^{-1}$
  and $\sfr_2^n(t):=\sfs^n_2\circ (\sfs_1^n)^{-1}$, we get
  \begin{displaymath}
    \lim_{n\to\infty}\sup_{t\in [0,1]} \sfd(\sigma_1(t),
    \sigma_2(\sfr^n_2(t)))=0,\qquad
    \lim_{n\to\infty}\sup_{t\in [0,1]} \sfd(\sigma_1(\sfr_1^n(t)),
    \sigma_2(t))=0.
  \end{displaymath}
  The lower semicontinuity of the length with respect to 
  uniform convergence yields
  $\ell:=\ell(\sigma_1)=\ell(\sigma_2)$ and therefore
  for every $0\le t'<t''\le 1$
  \begin{displaymath}
    \ell \liminf_{n\to\infty}\bigl( r_2^n(t'')-r_2^n(t')\bigr)=
    \lim_{n\to\infty} \int_{t'}^{t''}|(\sigma_2\circ \sfr_2^n)'|\,\d
    t\ge 
     \int_{t'}^{t''}|\sigma_1'|\,\d t=
     \ell (t''-t').    
  \end{displaymath}
  Choosing first $t'=t$ and $t''=1$ and then $t'=0$ and $t''=t$ we conclude
  that $\lim_n r_2^n(t)=t$ for every $t\in [0,1]$ and therefore
  $\sigma_1=\sigma_2$.
  
  Notice that $\Lipc X$ is a Lusin space, since $\Lipc X$ is a Borel subset
  of $\rmC([0,1];X)$.
  The restriction of $\sfi$ to $\Lipc X$ is thus a continuous and injective map 
  from the Lusin space $\Lipc X$ to the Hausdorff space $(\Curvesnonpara X,\tau_{\mathscr C})$
  (notice that the topology $\tau_{\mathscr C}$ is a priori weaker than
  the one induced by the restriction of $\sfi$ to $\Lipc X$).
  It follows by definition that $\Curvesnonpara X$ is Lusin. Now, Proposition~\ref{proplusin}(iii) yields that the
  restriction of $\sfi$ is a Borel isomorphism.
\end{proof}

\begin{lemma} [Borel regularity of $J$ and $\tilde J$] \label{lem:vivaBogachev} 
The map $J: \AC([0,1];X)\to  \mathcal{M}_+(X)$
is Borel, where $\AC([0,1];X)$ is endowed with the $\rmC([0,1];X)$
topology. In particular, if $(X,\sfd)$ is complete and separable
the map $\tilde J:\Curvesnonpara{X}\to \mathcal{M}_+(X)\setminus\{0\}$
is Borel and $\tilde J(\Gamma)$ is Souslin in $\cM_+(X)$ whenever $\Gamma$ is Souslin in $\Curvesnonpara{X}$.
\end{lemma}
\begin{proof} It is easy to check, using the formula $J\gamma=\gamma_\sharp(|\dot\gamma|\Leb{1}\res [0,1])$, that
$$
J\gamma=\lim_{n\to\infty}\sum_{i=0}^{n-1}\sfd(\gamma_{(i+1)/n},\gamma_{i/n})\delta_{\gamma_{i/n}} 
\qquad\text{weakly in $\cM_+(X)$}
$$
for all $\gamma\in \AC([0,1];X)$ (the simple details are left to the reader).
Since the approximating maps are continuous, we conclude that
$J$ is Borel. The Borel regularity of $\tilde J$ follows by Lemma~\ref{le:Lusin}
and the identity $\tilde J= J\circ\sfj$. Since $\tilde J$ is Borel, we can apply Proposition~\ref{proplusin}(iv)
  to obtain that $\tilde J$ maps
  Souslin sets to Souslin sets.
\end{proof}

\section{Modulus of families of non-parametric curves}\label{sec:weakgrad}

In this section we assume that $(X,\sfd)$ is a complete and separable metric space and that
$\mm\in {\mathcal M}_+(X)$.

In order to apply the results of the previous sections 
(with the topology $\tau$ induced by $\sfd$)  to families of non-parametric curves
we consider the canonical map $\tilde J:\Curvesnonpara{X}\to {\cM}_+(X)\setminus\{0\}$
of Definition~\ref{def:canonical}(c). In the sequel, for the sake of simplicity, we will not distinguish between $J$ and $\tilde J$,
writing $J\pgamma$ or $J[\gamma]=J\gamma$ (this is not a big abuse of notation, since $\tilde J$ is a quotient map).

Now we discuss the notion of $(p,\mm)$-modulus, for $p\in [1,\infty)$.
The $(p,\mm)$-modulus for families $\Gamma\subset\Curvesnonpara{X}$ of non-parametric curves is given by
\begin{equation} \label{eqn:mod22}
{ \rm Mod}_{p,\smm} (\Gamma) := \inf \left\{  \int_X g^p \, \d \mm \, : \, g \in \Ldp,\,\, \int_\pgamma  g \geq 1 \; \; 
\text{ for all } \pgamma \in \Gamma\right\}.
\end{equation}
We adopted the same notation ${\rm Mod}_{p,\smm}$ because the identity $\int_\pgamma g=\int_X g\,\d J\pgamma$ immediately gives
\begin{equation}\label{eq:J2}
{\rm Mod}_{p,\smm}(\Gamma)={\rm Mod}_{p,\smm}(J(\Gamma)).
\end{equation}

In a similar vein, setting $q=p'$, in the space $\Curvesnonpara{X}$ we can define plans with barycenter in $L^q(X,\mm)$
as Borel probability measures $\ppi$ in  $\Curvesnonpara{X}$ satisfying 
$$
\int_{\Curvesnonpara{X}}J\pgamma\,\d\ppi(\pgamma)=g\mm\qquad\text{for some $g\in L^q(X,\mm)$.}
$$
Notice that the integral in the left hand side makes sense because the Borel regularity of $J$ easily gives
that $\pgamma\mapsto J\pgamma(A)$ is Borel in $\Curvesnonpara{X}$ for all $A\in\BorelSets{X}$.
We define, exactly as in \eqref{eq:cppirhoq}, $c_q(\ppi)$ to be the $L^q(X,\mm)$ norm of the barycenter $g$.  Then, the same argument leading to
\eqref{eq:sup1} gives
\begin{equation}\label{eq:55}
\frac{\ppi(\Gamma)}{c_q(\ppi)}\leq {\rm Mod}_{p,\smm}(\Gamma)^{1/p}\qquad
\text{for all $\ppi\in\cP(\Curvesnonpara X)$ with barycenter in $L^q(X,\mm)$}
\end{equation}
for every universally measurable set $\Gamma$ in $\Curvesnonpara{X}$.

\begin{remark} [Democratic plans] \label{rem:demo}
{\rm In more explicit terms, Borel probability measures $\ppi$ in $\Curvesnonpara{X}$ with barycenter in $L^q(X,\mm)$ satisfy
\begin{equation}\label{eq:democ}
\int_0^1 (\e_t)_\sharp (|\dot\gamma_t|\ppi)\,\d t= g\mm\quad\text{for some $g\in L^q(X,\mm)$}
\end{equation}
when we view them as measures on nonconstant curves $\gamma\in\AC([0,1];X)$. For instance, in the particular case when $\ppi$ is concentrated on family of geodesics
parameterized with constant speed and with length uniformly bounded from below, the case $q=\infty$ corresponds to the class
of democratic plans considered in \cite{Lott-Villani}.}
\end{remark}

Defining $C_{p,\smm}(\Gamma)$ as the supremum
in the left hand side of \eqref{eq:55}, we can now use Theorem~\ref{tmain} to show that even in this case there is no duality gap.

\begin{theorem}\label{tmain1}
  For every $p>1$ and every Souslin set
  $\Gamma\subset\Curvesnonpara{X}$ with ${\rm Mod}_{p,\smm}(\Gamma)>0$ 
  there exists a $\ppi\in\cP\bigl(\Curvesnonpara{X}\bigr)$ with barycenter in $L^q(X,\mm)$,
concentrated on $\Gamma$ and satisfying $c_q(\ppi)={\rm Mod}_{p,\smm}(\Gamma)^{-1/p}$.
\end{theorem}
\begin{proof}  From Theorem~\ref{tmain} we deduce 
the existence of $\eeta\in\cP\bigl({\mathcal M}_+(X)\bigr)$ with barycenter in $L^q(X,\mm)$ concentrated on the Souslin set ${J}({\Gamma})$ and satisfying
$$
\frac{1}{c_q(\eeta)}={\rm Mod}_{p,\smm}({J}({\Gamma}))^{1/p}={\rm Mod}_{p,\smm}({\Gamma})^{1/p}.
$$
By a measurable selection theorem \cite[Theorem~6.9.1]{Bogachev}
we can find a $\eeta$-measurable map $f:J(\Gamma)\to\Curvesnonpara X$ such that 
$f(\mu)\in \Gamma\cap {J}^{-1}(\mu)$ for all $\mu\in J(\Gamma)$. 
The measure  $\ppi:=f_\sharp \eeta$  is concentrated on $\Gamma$ and the equality 
between the barycenters
$$
\int_{\Curvesnonpara X} J\pgamma\,\d\ppi(\pgamma)=\int \mu\,\d\eeta(\mu)
$$
gives $c_q(\ppi)=c_q(\eeta)$. 
\end{proof}

\section{Modulus of families of parametric curves}

In this section we still assume that $(X,\sfd)$ is a complete and separable metric space and
that $\mm\in {\mathcal M}_+(X)$. We consider a notion of $p$-modulus for \emph{parametric} curves, enforcing the
condition \eqref{eq:democ} (at least when Lipschitz curves are considered), and 
we compare with the non-parametric counterpart. To this aim, we introduce the continuous map
\begin{equation}
  \label{eq:14}
  M:\rmC([0,1];X)\to \cP(X),\qquad
  M(\gamma):=\gamma_\sharp \big(\Leb 1\res [0,1]\big).
\end{equation}
Indeed, replacing $ J\gamma=\gamma_\sharp(|\dot\gamma|\Leb{1}\res [0,1]$) 
with $M$ we can consider a ``parametric'' modulus of a family of
curves $\Sigma\subset \rmC([0,1];X)$ just by
evaluating $\Md(M(\Sigma))$. By Proposition~\ref{prop:prop}(vii),
if $\Sigma\subset \AC^\infty_c([0,1];X)$ then
\begin{equation}
  \label{eq:19}
  \Md(M(\Sigma))=0\quad\Longleftrightarrow\quad
  \Md( J(\Sigma))%=\Md(\sfi\Sigma)
  =0.
\end{equation}
On the other hand, things are more subtle when the speed is not constant.

\begin{definition} [$q$-energy and 
  parametric barycenter] \label{defalphaq1}
  Let $\rrho\in\cP\bigl(\rmC([0,1];X)\bigr)$ and $q\in [1,\infty)$. 
  We say that $\rrho$ has finite $q$-energy
  if $\rrho$ is concentrated on $\AC^q([0,1];X)$ and 
\begin{equation}
  \label{eq:2}
  %\int \cE_q(\gamma)\,\d\rrho(\gamma)=
  \int \int_0^1 |\dot\gamma_t|^q\,\d t\,\d\rrho(\gamma)<\infty.
\end{equation}
We say that $\rrho$ has parametric barycenter 
$h\in L^q(X,\mm)$ if 
\begin{equation}
  \label{eq:3}
  \int \int_0^1 f(\gamma_t)\,\d t\,\d\rrho(\gamma)=
  \int_X f\,h\,\d\mm\quad\forall f\in \rmC_b(X).
\end{equation}
\end{definition}

The finiteness condition \eqref{eq:2} and the concentration on $\AC^q([0,1];X)$ can be
also be written, recalling the definition \eqref{eq:4} of $\cE_q$, as follows:
$$
\int \cE_q(\gamma)\,\d\rrho(\gamma)<\infty.
$$
Notice also that the definition \eqref{eq:14} of $M$ gives that \eqref{eq:3} is equivalent to require the existence of a
constant $C\ge0$ such that 
\begin{equation}
  \label{eq:13}
  \int \int_X f \,\d M\gamma\,\d\rrho(\gamma)
  \le C\Big(\int_X f^p\,\d\mm\Big)^{1/p}
  \quad\forall f\in \rmC_b(X),\ f\ge 0.
\end{equation}
In this case the best constant $C$ in \eqref{eq:13} corresponds to 
$\|h\|_{L^q(X,\mm)}$ for $h$ as in \eqref{eq:3}.

\begin{remark}
  \label{rem:para-vs-nonpara}
  \upshape
  It is not difficult to check that a Borel probability measure $\rrho$ concentrated on a
  set $\Gamma \subset \AC^\infty([0,1];X)$ with $\rrho$-essentially bounded
  Lipschitz constants and parametric barycenter in $L^q(X,\mm)$ has
  also (nonparametric) barycenter in $L^q(X,\mm)$.  Conversely, if
  $\ppi\in\cP\bigl(\Curvesnonpara X\bigr)$ with barycenter in
  $L^q(X,\mm)$ and $\ppi$-essentially bounded length $\ell(\gamma)$,
  then $\sfj_\sharp\ppi$ has parametric barycenter in $L^q(X,\mm)$.
\end{remark}

Now, arguing as in the proof of Theorem~\ref{tmain1} (which provided existence of plans $\ppi$ in $\Curvesnonpara X$)
we can use a measurable selection theorem to deduce from our basic duality Theorem~\ref{tmain}
the following result.

{%\color{blue}
\begin{theorem}\label{tmain11}
  For every $p>1$ and every Souslin set
  $\Sigma\subset\rmC([0,1];X)$,  $\Md(M(\Sigma))>0$ is equivalent to the 
  existence of $\rrho\in\cP\bigl(\rmC([0,1];X)\bigr)$ concentrated on $\Sigma$ with parametric barycenter in $L^q(X,\mm)$.
\end{theorem}
}

Our next goal is to use reparameterizations to improve the parametric barycenter from
$L^q(X,\mm)$ to $L^\infty(X,\mm)$. To this aim, we begin by proving the Borel regularity of some
parameterization maps.
Let $h:X\to (0,\infty)$ be a Borel map with $\sup_X h<\infty$ and
for every $\sigma\in\rmC([0,1];X)$ let us set
\begin{equation}
  \label{eq:9bis}
  G(\sigma):=\int_0^1 
  {h(\sigma_r)}\,\d r,\qquad
  \sft_\sigma(s):=\frac{1}{G(\sigma)}\int_0^s {h(\sigma_r)}\,\d r:[0,1]\to [0,1].
\end{equation}
Since $\sft_\sigma$ is Lipschitz and $\sft_\sigma'>0$ $\Leb{1}$-a.e. in $(0,1)$, 
its inverse
$\sfs_\sigma:[0,1]\to [0,1]$ is absolutely continuous and we can define
\begin{equation}
  \label{eq:10}
  H:\AC([0,1];X)\to \AC([0,1];X),\qquad
  H\sigma(t):=\sigma(\sfs_\sigma(t)).
\end{equation}
Notice that $H\big(\Lipc X\big)\subset \AC_0([0,1];X)$.

\begin{lemma}
  \label{le:parameterization}
  If $h:X\to \R$ is a bounded Borel function, the map $G$ 
  in \eqref{eq:9bis} is Borel. If we assume, in addition, that
  $h>0$ in $X$, then also $\sft_\sigma$ in \eqref{eq:9bis} is Borel 
  and the map $H$ in \eqref{eq:10} is Borel and injective.
  \end{lemma}
\begin{proof} Let us prove first that the map
$$
\sigma\mapsto\tilde \sft_\sigma(t)=\int_0^t h(\sigma_r)\,\d r
$$
is Borel from $\rmC([0,1];X)$ to $\rmC([0,1])$ for any bounded Borel function $h:X\to\R$.
This follows by a monotone class argument (see for instance \cite[Theorem~2.12.9(iii)]{Bogachev}), since
class of functions $h$ for which the statement is true is a vector space containing bounded continuous functions
and stable under equibounded pointwise limits. By the continuity of the integral operator, 
the map $G$ is Borel as well. 

Now we turn to $H$, assuming that $h>0$. 
By Proposition~\ref{proplusin}(iii) it will be sufficient to show that the inverse of $H$, namely the map $\sigma\mapsto \sigma\circ\sft_\sigma$, is Borel. 
Since the map $(\sigma,\sft)\mapsto \sigma\circ\sft$ is continuous from $\rmC([0,1];X)\times\rmC([0,1])$ to $\rmC([0,1];X)$, the Borel regularity
of the inverse of $H$ follows by the Borel regularity of $\sigma\mapsto\sft_\sigma$.
 \end{proof}

\begin{theorem}
  \label{thm:repara}
  Let $q\in (1,\infty)$ and $p=q'$. If $\rrho\in \cP\bigl(\rmC([0,1];X)\bigr)$ 
  has finite $q$-energy and parametric barycenter $h\in L^\infty(X,\mm)$, then 
  $\ppi=\sfi_\sharp \rrho$ has barycenter in $L^q(X,\mm)$
  and
  \begin{equation}
    \label{eq:8}
    c_q(\ppi)\le \Big(\int\cE_q(\gamma)\,\d\rrho(\gamma)\Big)^{1/q}
    \|h\|^{1/p}_{L^\infty(X,\mm)}.
  \end{equation}
  Conversely, if $\ppi\in\cP\bigl(\Curvesnonpara X\bigr)$ has
  barycenter in $L^q(X,\mm)$ and $\ppi$-essentially bounded length $\ell(\gamma)$,
  concentrated on a Souslin set $\Gamma\subset\Curvesnonpara X$, there exists 
  $\rrho\in\cP\bigl(\rmC([0,1];X)\bigr)$ with finite $q$-energy
  and parametric barycenter in $L^\infty(X,\mm)$ 
  concentrated in a Souslin set contained in $[\sfj(\Gamma)]$.
  
  More generally, 
  let $\ssigma\in\cP\bigl(\rmC([0,1];X)\bigr)$ be 
  concentrated on a Souslin set $\Gamma\subset\AC^\infty([0,1];X)$, 
  with parametric barycenter in $L^q(X,\mm)$
  and with $\ssigma$-essentially bounded Lipschitz constants.
  Then there exists $\rrho\in\cP\bigl(\rmC([0,1];X)\bigr)$ with finite $q$-energy
  and parametric barycenter in $L^\infty(X,\mm)$ 
  concentrated on a Souslin set contained in $[\Gamma]$.
\end{theorem}
\begin{proof} 
  Notice that for every nonnegative Borel $f$ there holds
  \begin{align*}
    \iint_\pgamma &f\,\d\ppi(\pgamma)= \iint_0^1 f(\gamma_t)\,|\dot
    \gamma_t|\,\d t\,\d\rrho(\gamma) \le
    \Big(\int\cE_q\,\d\rrho\Big)^{1/q} \Big(\iint_0^1
    f^{p}(\gamma_t)\,\d t\,\d\rrho(\gamma)\Big)^{1/p}
    \\&\le 
    \Big(\int\cE_q\,\d\rrho\Big)^{1/q}
    \Big(\int_X f^p\,h\,\d\mm\Big)^{1/p}\le 
    \Big(\int\cE_q\,\d\rrho\Big)^{1/q}
    \|h\|^{1/p}_{L^\infty(X,\mm)}\|f\|_{L^{p}(X,\mm)},
  \end{align*}
  so that \eqref{eq:8} holds.
  
  Let us now prove the last statement from $\ssigma$ to $\rrho$, since the ``converse'' statement from $\ppi$ to $\rrho$ simply
  follows by applying the last statement to $\ssigma:=\sfj_\sharp \ppi$ and recalling 
  Remark~\ref{rem:para-vs-nonpara}.
    Let $g\in L^q(X,\mm)$ be the parametric 
  barycenter of $\ssigma$ and let us set 
  $h:=1/(\eps\lor g)$, with $\eps>0$ fixed.
  Up to a modification of $g$ in 
  a $\mm$-negligible set, it is not restrictive to assume
  that $h$ is Borel and with values in $(0,1/\eps]$, so that the corresponding 
  maps $G$ and $H$ defined 
  as in \eqref{eq:9bis} and \eqref{eq:10} are Borel.

  We set 
  $\hat{\rrho}:=z^{-1}\,G(\cdot)\ssigma$, where 
  $z\in (0,1/\eps]$ is the normalization constant 
  $\int G(\gamma)\,\d\ssigma(\gamma)$.
  Let us consider the inverse $\sfs_\sigma:[0,1]
  \to [0,1]$ of the map $\sft_\sigma$ in \eqref{eq:9bis}, which is absolutely continuous for every $\sigma$
  and the corresponding transformation $H\sigma$ in \eqref{eq:10}.
  We denote by $L$ the $\ssigma$-essential supremum of 
  the Lipschitz constants of the curves in $\Gamma$. 
  Notice that for $\ssigma$-a.e.~$\sigma$
  \begin{equation}
    \label{eq:11}
    |(H\sigma)'|(t)\le 
    L\sfs_\sigma'(t)=\frac{L \,G(\sigma)}{h(H\sigma(t))}
   \qquad\text{$\Leb{1}$-a.e. in $(0,1)$,}
  \end{equation}
  and that for every nonnegative Borel function $f$ one has
  \begin{displaymath}
    \int_0^1 f(H\sigma(t))\,\d t=\int_0^1 f(\sigma(\sfs_\sigma(t)))\,\d t=
    \int_0^1 f(\sigma(s))\sft'_\sigma(s)\,\d s=
      \frac 1{G(\sigma)}
      \int_0^1 f(\sigma(s))   h(\sigma(s))\,\d s,
  \end{displaymath}
   so that choosing $f=h^{-q}$ and using the inequality $G\leq 1/\eps$ yields
  \begin{equation}
    \label{eq:12}
    \cE_q(H \sigma)\le 
    L^q G^q(\sigma)
    \int_0^1
    h^{-q}(H\sigma(t))
    \,\d t\leq \frac{L^q}{\eps^{q-1}}
    \int_0^1
    h^{1-q}(\sigma(s))% \,|(H \sigma)'|(t)
    \,\d s.
  \end{equation}
    Now we set $\rrho:=H_\sharp\hat{\rrho}$ and notice that, by construction, $\rrho$ is concentrated on 
    the Souslin set $H(\Gamma)\subset [\Gamma]$.
   Integrating the $q$-energy with respect to $\rrho$ we obtain
  \begin{align*}
    \int \cE_q(\theta)\,\d\rrho(\theta)&=
    \int \cE_q(H\sigma)\,\d\hat{\rrho}(\sigma)\leq
    \frac{L^q}{z\eps^{q-1}}\int G(\sigma)\int_0^1
    h^{1-q}(\sigma(s))\,\d s\,\d\ssigma(\sigma)\le 
    \frac{L^{q}}{z\eps^q} \int_X gh^{1-q}\,\d\mm
    \\&=
    \frac{L^{q}}{z\eps^q} \int_X g(\eps\lor g)^{q-1}\,\d\mm<\infty,
  \end{align*}
  thus obtaining %(since $gh\leq 1$) 
  that $\rrho$ has finite $q$-energy.
  Similarly
  \begin{align*}
    \int \int_0^1 f(\theta(t))\,\d t\,\d\rrho(\theta)&=
    \int \int_0^1 f(H\sigma(t))\,\d t\,d\hat{\rrho}(\sigma)=
    \frac 1z\int \int_0^1 f(\sigma(s))h(\sigma(s))\,\d s\,\d\ssigma(\sigma)
    \\&=\frac 1z\int_X fgh\,\d\mm.
  \end{align*}
  Since $gh\leq 1$, this shows that $\rrho$ has parametric
  barycenter in $L^\infty(X,\mm)$.
\end{proof}
In the next corollary, in order to avoid further measurability issues, we state our result with the
inner measure
$$
\mu_*(E):=\sup\left\{\mu(B):\ \text{$B$ Borel, $B\subset E$}\right\}.
$$
This formulation is sufficient for our purposes.

\begin{corollary}
  \label{cor:almost}
  A Souslin set $\Gamma\subset \Curvesnonpara X$ is $\Md$-negligible if and only if 
  $\rrho_*([\sfj\Gamma])=0$ for every $\rrho\in \cP\bigl(\rmC([0,1];X)\bigr)$ concentrated on $\AC^q([0,1];X)$ and
  with parametric barycenter in $L^\infty(X,\mm)$.
  \end{corollary}
\begin{proof}
  Let us first suppose that $\Gamma$ 
  is $\Md$-negligible and let us denote
  by $h\in L^\infty(X,\mm)$ the parametric barycenter of 
  $\rrho$ and let us prove that $
 \rrho_*([\sfj\Gamma])=0$. Since $\rrho$ is concentrated on $\AC^q([0,1];X)$ we
 can assume with no loss of generality (possibly restricting $\rrho$ to the class of
 curves $\sigma$ with $\cE_q(\sigma)\leq n$ and normalizing) that $\rrho$ has finite $q$-energy.
  We observe that if $\sigma\in\AC([0,1];X)$ and $f:X\to [0,\infty]$ is Borel, there holds
  \begin{equation}
    \label{eq:17}
    \int_0^1 f(\sigma(t))|\dot\sigma(t)|\,\d t\le 
    \Big(\int_0^1 f^p(\sigma(t))\,\d t\Big)^{1/p}
    \Big(\cE_q(\sigma)\Big)^{1/q}.
  \end{equation}
  If $f$ satisfies
$$
\int_\pgamma f\geq 1\qquad \forall \pgamma\in\Gamma
$$
we obtain that $\int_\sigma f\geq 1$ for all $\sigma\in [\sfj\Gamma]$. We can now integrate w.r.t. $\rrho$
and use \eqref{eq:17} to get
\begin{align}
  \notag &\rrho_*([{\sf j}\Gamma])\leq
  \Big(\int \int_0^1 f^p(\sigma(t))\,\d t\,\d\rrho(\sigma)\Big)^{1/p}
  \Big(\int \cE_q(\sigma)\,\d\rrho(\sigma)\Big)^{1/q}
  \\&=\label{eq:elementary} 
  \Big(\int_X f^p\, h\,\d\mm\Big)^{1/p}
  \Big(\int \cE_q(\sigma)\,\d\rrho(\sigma)\Big)^{1/q}
  \le \|f\|_{p}
    \|h\|_{\infty}^{1/p}
    \Big(\int \cE_q(\sigma)\,\d\rrho(\sigma)\Big)^{1/q}.
  \end{align}
By minimizing with respect to $f$ we obtain that 
$\rrho_*([{\sf j}\Gamma])=0$. 

Conversely, suppose that $\Md(\Gamma)>0$; possibly passing to a smaller set, by the countable
subadditivity of $\Md$ we can assume that $\ell$ is bounded on $\Gamma$: then by Theorem~\ref{tmain1} there exists 
$\ppi\in\cP\bigl(\Curvesnonpara X\bigr)$ with barycenter in $L^q(X,\mm)$ concentrated on $\Gamma$
and therefore the boundedness of $\ell$ allows to apply the final statement of Theorem~\ref{thm:repara}
to obtain $\rrho\in \cP\bigl(\rmC([0,1];X)\bigr)$ with finite $q$-energy, parametric barycenter in $L^\infty(X,\mm)$ and
concentrated on a Souslin subset of $[{\sf j}\Gamma]$.
\end{proof}

\begin{corollary}
  \label{cor:almost2}
  Let $\Gamma\subset\AC^\infty([0,1];X)$ be a Souslin set 
  such that 
  $\rrho_*\big([\Gamma]\big)=0$ for every plan 
  $\rrho\in \cP(\rmC([0,1];X))$ concentrated on $\AC^q([0,1];X)$ and
  with parametric barycenter in $L^\infty(X,\mm)$.
  Then 
  $M(\Gamma)$ is $\Md$-negligible.  
\end{corollary}
\begin{proof}
  Suppose by contradiction that $\Md(M(\Gamma))>0$; possibly passing to a smaller set, by the countable
  subadditivity of $\Md$ we can assume that 
  $\Lip$ is bounded on $\Gamma$. By Theorem~\ref{tmain11} there exists 
  $\ppi\in \cP\bigl(\rmC([0,1];X)\bigr)$ with parametric barycenter in $L^q(X,\mm)$ concentrated on $\Gamma$.
  The boundedness of $\Lip$  on $\Gamma$ allows to apply the second part of Theorem~\ref{thm:repara}
  to obtain $\rrho\in \cP\bigl(\rmC([0,1];X)\bigr)$ with parametric barycenter in $L^\infty(X,\mm)$, finite
  $q$-energy and concentrated on a Souslin subset of $[\Gamma]$.
\end{proof}

\section{Test plans and their null sets}

In this section we will assume that $(X,\sfd)$ is a complete and
separable metric space and $\mm\in \cM_+(X).$
The following notions 
have already been used in \cite{Ambrosio-Gigli-Savare11} ($q=2$)
and \cite{Ambrosio-Gigli-Savare12} (in connection with the Sobolev
spaces with gradient in $L^p(X,\mm)$, with $q=p'$; see also \cite{AmbrosioDiMarino12} in connection
with the $BV$ theory).

\begin{definition} [$q$-test plans and negligible sets] \label{defalphaq}
Let $\rrho\in\Probabilities{\rmC([0,1];X)}$ and $q\in [1,\infty]$. We say that $\rrho$ is
a $q$-test plan if
\begin{itemize}
\item[(i)] $\rrho$ is concentrated on $\AC^q([0,1];X)$;
\item[(ii)] there exists a constant $C=C(\rrho)>0$ satisfying $(\e_t)_\sharp\rrho\leq C\mm$ for all $t\in [0,1]$.
\end{itemize}
We say that a universally measurable set $\Gamma\subset\rmC([0,1];X)$ is $q$-negligible if 
$\rrho(\Gamma)=0$ for all $q$-test plans $\rrho$.
\end{definition}

Notice that, by definition, $\rmC([0,1];X)\setminus\AC^q([0,1];X)$ is $q$-negligible.
The lack of invariance of these concepts, even under bi-Lipschitz reparameterizations is due to condition (ii), which is imposed at any given time
and with no averaging (and no dependence on speed as well). Since condition (ii) is more restrictive compared
for instance to the notion of democratic test plan of \cite{Lott-Villani} (see Remark~\ref{rem:demo}), this means 
that sets of curves have higher chances of being negligible w.r.t. this notion, as the next elementary example shows. 

We now want to relate null sets according to Definition~\ref{defalphaq} to null
sets in the sense of $p$-modulus.
Notice first that in the definition of $q$-negligible set we might consider only plans $\rrho$ satisfying the stronger condition
\begin{equation}\label{eqn:EL} \esssup \{ \cE_q( \sigma ) \} < \infty 
\end{equation}
because any $q$-test plan can be monotonically be approximated by $q$-test plans satisfying this
condition. Arguing as in the proof of \eqref{eq:elementary}
we easily see that 
\begin{equation}\label{eq:bl2}
\text{$\Gamma\subset\Curvesnonpara X$ ${\rm Mod}_{p,\smm}$-negligible}
\qquad\Longrightarrow\qquad 
\text{${\sfi}^{-1}(\Gamma)$ $q$-negligible.}
\end{equation} 
The following simple example shows that the implication can't be
reversed, namely sets whose images under ${\sfi}^{-1}$ are $q$-negligible 
need not be ${\rm Mod}_{p,\smm}$-null.

\begin{example}
Let $X=\R^2$, $\sfd$ the Euclidean distance, $\mm=\Leb{2}$. The family of parametric segments
$$
\Sigma=\left\{\gamma^x:\ x\in [0,1]\right\} \subset \AC([0,1];\R^2) 
$$
with $\gamma^x_t=(x,t)$ is $q$-negligible for any $q$, but ${\sf i}(\Sigma)$ has $p$-modulus
equal to 1.
\end{example}

In the previous example the implication fails because the trajectories $\gamma^x$ fall, at any given time $t$, into a $\mm$-negligible
set, and actually the same would be true if this concentration property holds at some fixed time. It is tempting to imagine that
the implication is restored if we add to the initial family of curves all their reparameterizations (an operation
that leaves the $p$-modulus invariant). However, since any reparameterization fixes the endpoints, even this
fails. However, in the following, we will see that the implication 
$$
\text{$\Gamma$ $q$-negligible}\qquad\Longrightarrow\qquad {\rm Mod}_{p,\smm}(\sfi(\Gamma))=0
$$
could be restored if we add some structural assumptions on $\Gamma$ (in particular a ``stability'' condition); 
the collections of curves we are mainly interested in are those connected with the theory of Sobolev spaces in \cite{Ambrosio-Gigli-Savare11}, 
\cite{Ambrosio-Gigli-Savare12}, and we will find a new proof of the fact that if we define weak upper gradients according to the two 
notions, the Sobolev spaces are eventually the same.

We now fix some additional notation: for $\J=[a,b]\subset [0,1]$ we define the ``stretching'' map ${\sf s}_\J:\AC([0,1];X)\to\AC([0,1];X)$,
mapping $\gamma$ to $\gamma\circ s_\J$, where $s_\J:[0,1]\to [a,b]$ is the affine map with $s_\J(0)=a$ and $s_\J(1)=b$. Notice that
this map acts also in all the other spaces $\AC^q$, $\AC_0$, $\AC^\infty_c$ of parametric curves we are considering.

\begin{definition}[Stable and invariant sets of curves]\null\ 
\begin{itemize}
\item[(i)] We say that $\Gamma\subset\{\gamma\in\AC([0,1];X):\ \ell(\gamma)>0\}$ is invariant under constant speed reparameterization 
if $\sfk\gamma\in\Gamma$ for all $\gamma\in\Gamma$;
\item[(ii)] We say that $\Gamma\subset\AC([0,1];X)$ is $\sim$-invariant 
if $[\gamma]\subset\Gamma$ for all $\gamma\in\Gamma$; 
\item[(iii)] We say that $\Gamma\subset\AC([0,1];X)$ is stable
if for every $\gamma\in\Gamma$ there exists $\ep\in (0,1/2)$ such that $s_{\J} \gamma\in\Gamma$ whenever $\J=[a,b]\subset [0,1]$ 
and $|a|+|1-b| \leq \ep$.
\end{itemize}
\end{definition}

{%\color{blue}
The following theorem provides key connections between $q$-negligibility and ${\rm Mod}_{p,\smm}$-negligibility, both in the nonparametric
sense (statement (i)) and in the parametric case (statement (ii)), for stable sets of curves. 
}

\begin{theorem}\label{teo:popen} Let $\Gamma\subset\AC([0,1];X)$ be a Souslin and stable set of curves.
\begin{itemize}
\item[(i)] If, in addition, $\ell(\gamma)>0$ for all $\gamma\in\Gamma$
  and $\Gamma$ is both $\sim$-invariant and invariant under constant 
speed reparameterization, then $\Gamma$ is $q$-negligible if and only if $ J(\Gamma)$ is ${\rm Mod}_{p,\smm}$-negligible in ${\mathcal M}_+(X)$
(equivalently, $\sfi(\Gamma)$ is ${\rm Mod}_{p,\smm}$-negligible in
$\Curvesnonpara{X}$).
\item[(ii)] If $\Gamma$ is $q$-negligible and
$[\Gamma\cap\AC^\infty([0,1];X)]\subset\Gamma$,
then $M\bigl(\Gamma\cap\AC^\infty([0,1];X)\bigr)$ is ${\rm Mod}_{p, \smm}$-negligible
in ${\mathcal M}_+(X)$. If $\Gamma$ is also $\sim$-invariant then the converse holds, too.
\end{itemize}
\end{theorem}

\begin{proof} (i) The proof of the nontrivial implication, from positivity of
${\rm Mod}_{p,\smm}(J(\Gamma))$ to $\Gamma$ being not $q$-negligible
is completely analogous to the proof of (ii), given below, by applying Corollary~\ref{cor:almost} to $\sfi (\Gamma)$ 
in place of  Corollary~\ref{cor:almost2} to $\Gamma\cap\AC^\infty([0,1];X)$ and the same rescaling technique. Since we will only need (ii) in 
the sequel, we only give a detailed proof of (ii).

(ii) Let us prove that the positivity of ${\rm Mod}_{p,\smm} \bigl(M(\Gamma\cap\AC^\infty([0,1];X))\bigr)$ 
implies that $\Gamma$ is not $q$-negligible. Since $\Gamma\cap\AC^\infty([0,1];X)$ is stable, we can assume the existence of
$\ep\in (0,1/2)$ such that $s_{\J} \gamma\in\Gamma$ whenever $\J=[a,b]\subset [0,1]$ 
and $|a|+|1-b| \leq \ep$.
 
By applying Corollary~\ref{cor:almost2} to 
$\Gamma\cap\AC^\infty([0,1];X)$ we obtain the existence of $\rrho\in\cP\bigl(\AC^q([0,1];X)\bigr)$ 
concentrated on a Souslin subset of $[\Gamma\cap\AC^\infty([0,1];X)]$, and then on $\Gamma$,
with $L^\infty$ parametric barycenter, i.e. such that
\begin{equation}\label{eqn:boundinfinit}
\int_0^1 (\e_t)_{\sharp} \rrho \,\d t\leq C \mm \qquad \text{for some }C>0.
\end{equation}

Let's define a family of reparameterization maps $F_{\ep}^{\tau}:\AC^q([0,1];X) \to \AC^q([0,1];X)$: 
\begin{equation}\label{eqn:reparam}
F_{\ep}^{\tau} \gamma (t)= \gamma \Bigl( \frac{t+\tau}{1+\ep} \Bigr) \qquad t\in [0,1],\qquad\forall\gamma \in \AC^q([0,1];X), \; \; \forall \: \tau \in [0,\eps].
\end{equation}
Let us consider now the measure
$$ \rrho_\ep =\frac 1{\ep} \int_0^{\ep} (F_{\ep}^{\tau})_{\sharp} \rrho \, d \tau.$$
We claim that $\rrho_\ep$ is a $q$-plan: it is clear that $\rrho_\ep$ is a probability measure on $\AC^q([0,1];X)$, and so we have to check only the marginals at every time:
 \begin{align*}
 (\e_t)_{\sharp} \rrho_\ep &= \frac 1{\ep} \int_0^{\ep} (\e_t)_{\sharp} \bigl( (F_{\ep}^{\tau})_{\sharp} \rrho  \bigr) \, d \tau =  \frac {1}{\ep} \int_0^{\ep} (\e_{\frac {t+\tau}{1+\ep}})_{\sharp} \rrho  \, d \tau \\ 
 &= \frac {1+\ep}{\ep} \int_{\frac {t}{1+\ep}}^{\frac {t+\ep}{1+\ep}} (\e_s)_{\sharp} \rrho  \, d s  \leq \frac {1+\ep}{\ep} \int_0^1 (\e_s)_{\sharp} \rrho  \, d s \leq C\frac {1+\ep}{\ep} \mm  \qquad \text{ for all }t \in [0,1].
  \end{align*}
Now we reach the absurd if we show that $\rrho_\ep$ is concentrated on $\Gamma$; in order to do so it is sufficient to 
notice that $F_{\ep}^{\tau}=s_\J$ with $\J=\J_\ep^\tau=[ \frac{ \tau}{1+\ep}, \frac{1+\tau}{1+\ep}]$ and $\tau\in [0,\ep]$. 

Now, if we assume also that $[\Gamma] \subset\Gamma$, then we know that for all $\gamma \in \Gamma$ the curve
$\eta:=\gamma \circ \sfs_1^{-1}$ belongs to $\Gamma \cap\AC^\infty([0,1];X)$, where $\sfs_1:[0,1]\to [0,1]$ is the parameterization defined in 
the proof of Proposition~\ref{prop:tilli}. We recall that by definition we have 
$(1+\ell(\gamma))\sfs_1'(t)= 1+| \dot{\gamma}_t |$
for $\Leb{1}$-a.e. $t$; in particular, the change of variables formula gives
\begin{equation}\label{eqn:chvar}
\int_0^1 (1+|\dot{\gamma_t}|)g( \gamma_t) \,\d t = (1+\ell(\gamma))\int_0^1 g( \eta_s) \, \d s \qquad \forall g:X\to [0,\infty]\text{ Borel.}
\end{equation}
We suppose that $M\bigl(\Gamma\cap\AC^\infty([0,1];X)\bigr)$ is ${\rm Mod}_{p, \smm}$-negligible; this gives us $f\in{\cal L}^p_+(X,\mm)$ such that 
\begin{equation}\label{eqn:int11}
 \int_0^1 f(\eta_s) \, \d s = \infty \qquad \forall \eta \in \Gamma \cap \AC^{\infty}([0,1];X). 
\end{equation}
Now given any $q$-plan $\ppi$ we have that
\begin{align}\nonumber \int_\Gamma\int_0^1 (|\dot{\gamma_t}| +1)f(\gamma_t) \, \d t \, \d \ppi(\gamma) & 
\leq \left( \iint_0^1 (|\dot{\gamma_t}| + 1)^q \d t \, \d \ppi(\gamma) \right)^{1/q} \left( \iint_0^1 f^p(\gamma_t) \d t \, \d \ppi(\gamma) \right)^{1/p} \\ \label{eqn:int12}
 & \leq  \left( \Bigl(\int \cE_q \, \d \ppi \Bigr)^{1/q} + 1 \right) \left( C(\ppi) \cdot \int_X f^p \, \d \mm \right)^{1/p} < \infty 
\end{align}
Now, using \eqref{eqn:int11}  and \eqref{eqn:chvar} with $g=f$ give
$\int_0^1 (|\dot{\gamma_t}| +1)f(\gamma_t) \, \d t=\infty$ for all $\gamma\in\Gamma$,
so that \eqref{eqn:int12} gives that $\ppi(\Gamma)=0$. Since $\ppi$ is arbitrary, $\Gamma$ is
$q$-negligible.
\end{proof}

\begin{remark}\label{rem:implic} \upshape 
We note that the proof shows that if $\Gamma$ is $\sim$-invariant and $M\bigl(\Gamma\cap\AC^\infty([0,1];X)\bigr)$ is 
${\rm Mod}_{p, \smm}$-negligible in ${\mathcal M}_+(X)$,  then $\Gamma$ is $q$-negligible, independently of the stability
assumption that we used in the converse implication.
\end{remark}

\section{Weak upper gradients}

As in the previous sections, $(X,\sfd)$ will be a complete and
separable metric space and $\mm\in \cM_+(X).$

Recall that a Borel function $g:X\to [0,\infty]$ is an upper gradient of $f:X\to\R$ if
\begin{equation} \label{eqn:ugin}
|f(\pgamma_{fin})-f(\pgamma_{ini})|\leq\int_\pgamma g
\end{equation}
holds for all $\pgamma\in\Curvesnonpara{X}$.
Here, the curvilinear integral $\int_\pgamma g$ is given by $\int_J g(\gamma_t)|\dot\gamma_t|\,\d t$, where $\gamma:J \to X$ is any
parameterization of the curve $\pgamma$ (i.e., $\pgamma=\sfi\gamma$, and one can canonically take $\gamma=\sfj\pgamma$).
It follows from Proposition~\ref{prop:tillibis} that the upper gradient property can be equivalently written in the form
$$
|f(\pgamma_{fin})-f(\pgamma_{ini})|\leq\int_X g\,\d J\pgamma.
$$
Now we introduce two different notions of Sobolev function and a corresponding notion of $p$-weak gradient; the first one was first given in \cite{Shanmugalingam00} while the second one in \cite{Ambrosio-Gigli-Savare11} for $p=2$ and in \cite{Ambrosio-Gigli-Savare12} for general exponent. When discussing the corresponding
notions of (minimal) weak gradient we will follow the terminology of  \cite{Ambrosio-Gigli-Savare12}.

\begin{definition}[$N^{1,p}$ and $p$-upper gradient] Let $f$ be a $\mm$-measurable and $p$-integrable function on $X$.
We say that $f$ belongs to the space $N^{1,p}(X, \sfd, \mm) $ if there exists  
$g\in {\cal L}_+^p(X,\mm)$ such that \eqref{eqn:ugin} is satisfied for $\Md$-a.e. curve $\pgamma$.
\end{definition}

Functions in $N^{1,p}$ have the important Beppo-Levi property of being absolutely continuous along $\Md$-a.e. curve $\pgamma$  
(more precisely, this means $f\circ\sfj\pgamma\in\AC([0,1])$), see
\cite[Proposition~3.1]{Shanmugalingam00}. Because of the implication \eqref{eq:bl2}, functions in $N^{1,p}(X,\sfd,\mm)$ belong to the Sobolev space defined below
(see \cite{Ambrosio-Gigli-Savare11}, \cite{Ambrosio-Gigli-Savare12}) where \eqref{eqn:ugin} is required for $q$-a.e. curve $\gamma$.

\begin{definition}[$W^{1,p}$ and $p$-weak upper gradient] Let $f$ be a $\mm$-measurable and $p$-integrable function on $X$.
We say that 
 $f$ belongs to the space $W^{1,p}(X, \sfd, \mm) $ if there exists $g\in {\cal L}_+^p(X,\mm)$ such that
$$
|f(\gamma_1)-f(\gamma_0)|\leq\int_0^1g(\gamma_t)|\dot\gamma_t|\,\d t
$$ 
is satisfied for $q$-a.e. curve $\gamma\in\AC^q([0,1];X)$.
\end{definition}

We remark that there is an important difference between the two definitions, namely the first one is a priori not invariant if we change the function 
$f$ on a $\mm$-negligible set, while the second one has this kind of invariance, because for any $q$-test plan $\rrho$, any $\mm$-negligible Borel
set $N$ and any $t\in [0,1]$ the set $\{\gamma\;:\;\gamma_t\in N\}$ is $\rrho$-negligible. Associated to these two notions are the minimal $p$-upper gradient
and the minimal $p$-weak upper gradient, both uniquely determined up to $\mm$-negligible sets 
(for a more detailed discussion, see \cite{Ambrosio-Gigli-Savare12, Shanmugalingam00}).

As an application of Theorem~\ref{teo:popen}, we show that these two notions are essentially equivalent modulo the choice
of a representative in the equivalence class: more precisely, for any $f\in  W^{1,p}(X,\sfd,\mm)$
there exists a $\mm$-measurable representative $\tilde f$ of $f$ which belongs to $N^{1,p}(X,\sfd,\mm)$.
This result is not new, because in \cite{Ambrosio-Gigli-Savare11} and \cite{Ambrosio-Gigli-Savare12} the equivalence has already
been shown. On the other hand, the proof of the equivalence in \cite{Ambrosio-Gigli-Savare11} and \cite{Ambrosio-Gigli-Savare12} is by no means elementary, it passes through
the use of tools from the theory of gradient flows and optimal transport theory and it provides the equivalence with another relevant notion of ``relaxed'' gradient 
based on the approximation through Lipschitz functions.
We provide a totally different proof, using the results proved in this paper about negligibility of sets of curves.

In the following theorem we provide, first,  existence of a ``good representative'' of $f$. Notice that the standard
theory of Sobolev spaces provides existence of this representative via approximation with Lipschitz functions.

\begin{theorem} [Good representative]\label{lem:goodcont} 
  Let $f:X \to \R$ be a Borel function and let us set
  $$
     \Gamma= \bigl\{ \gamma \in\AC^\infty([0,1];X) \; : \; f \circ \gamma \text{ has
     a continuous representative $f_\gamma:[0,1]\to\R$}\bigr\}.
  $$
  If $\Md\bigl(M(\AC^\infty([0,1];X)\setminus \Gamma)\bigr)=0$ there exists a $\mm$-measurable representative $\tilde{f}:X\to\R$ of $f$
  satisfying
   \begin{equation}\label{eq:Mneg}
   \Md \bigl( M (\{\gamma\in\Gamma\; : \; \tilde{f}\circ\gamma\not\equiv f_\gamma \} ) \bigr) = 0.
   \end{equation}
  In particular 
  \begin{itemize}
   \item[(i)] for $q$-a.e. curve $\gamma$ there holds $\tilde{f} \circ \gamma \equiv f_{\gamma}$;
 \item[(ii)] for $\Md$-a.e. curve $\pgamma$ there holds $\tilde{f} \circ \sfj \pgamma \equiv f_{\sfj \pgamma}$.
\end{itemize}
\end{theorem}
\begin{proof} Let us set $\tilde\Gamma:=\AC^{\infty}([0,1];X)\setminus \Gamma$, so that our assumption reads  $\Md(M(\tilde\Gamma))=0$. 
Notice first that the (ii) makes sense because $f_{\sfj\pgamma}$ exists for
$\Md$-a.e. curve $\pgamma$ thanks to \eqref{eq:19} and
$\Md( M (\tilde \Gamma\cap\AC^\infty_c([0,1];X)) )=0$
(also, constant curves are all contained in $\Gamma$). 
Also, (i) makes sense thanks to Remark~\ref{rem:implic} and to the fact that the defining property of $\Gamma$ is $\sim$-invariant.

\smallskip
\noindent Step 1. (Construction of a good set $\Gamma_g$ of curves).
Since we have $\Md(M(\tilde\Gamma))=0$, there exists $h\in {\cal L}_+^p(X,\mm)$ 
such that $\int_0^1 h \circ \sigma = \infty$ for every $\sigma \in \tilde{\Gamma}$. 
Starting from $\Gamma$ and $h$, we can define the set $\Gamma_g= \bigl\{\eta\in\Gamma\;:\; \int_0^1 h\circ \eta <\infty\bigr\} $ of ``good'' curves, satisfying the following three conditions:
\begin{itemize}
\item[(a)] $f\circ\eta$ has a continuous representative for all $\eta\in\Gamma_g$;
\item[(b)] $\int_0^1 h \circ \eta <\infty$ for all $\eta\in\Gamma_g$;
\item[(c)] $M\bigl(\AC^\infty([0,1];X)\setminus\Gamma_g\bigr)$ is $\Md$-negligible.
\end{itemize}
Indeed, properties (a) and (b) follow easily by definition, while (c) follows by the inclusion
$$
M\bigl(\AC^\infty([0,1];X)\setminus\Gamma_g\bigr)\subset
M\bigl(\AC^\infty([0,1];X)\setminus\Gamma\bigr)\cup\bigl\{\mu:\ \int_X h\,\d\mu=\infty\bigr\}.
$$

\noindent Step 2. (Construction of $\tilde f$).
 For every point $x \in X$ we consider the set of pairs good curves-times that pass through $x$ at time $t$:
$$\Theta_x = \{ (\eta,t)\in\Gamma_g\times [0,1] \; : \; \eta(t)=x \},$$
and, thanks to property (a) of $\Gamma_g$, we can partition this set according to the value of the continuous representative $f_\eta$ at $t$:
$$\Theta_x = \bigcup_{r \in \R} \Theta_x^r  \qquad\text{with}\qquad \Theta_x^r= \{ (\eta,t) \in \Theta_x \; : \; f_\eta (t)=r \}. $$
Now, the key point is that for every $x\in X$ there exists at most one $r$ such $\Theta_x^r$ is not empty.
Indeed, suppose that $r_1\neq r_2$ and that there exist $(\eta_1,t_1)\in\Theta_x^{r_1}$, 
$(\eta_2,t_2)\in\Theta_x^{r_2}$, so that $r_1=f_{\eta_1}(t_1) \neq f_{\eta_2}(t_2)=r_2$;
since $\eta_1,\, \eta_2 \in \Gamma_g$, property (b) of $\Gamma_g$ gives
\begin{equation}\label{eq:starstar}
\int_0^1 h\circ \eta_1 \,\d t + \int_0^1 h \circ \eta_2  \,\d t< \infty.
\end{equation}
Suppose to fix the ideas that $t_1>0$ and $t_2<1$ (otherwise we reverse time for one curve, or both, in the following argument). 
Now we create a new curve $\eta_3\in\AC^{\infty}([0,1];X)$ by concatenation:
$$
\eta_3(s):= \begin{cases} \eta_1(2st_1) & \text{ if }s\in [0,1/2], \\ \eta_2(1-2(1-s)(1-t_2)) & \text{ if }s\in [1/2,1].  
\end{cases}$$
This curve is clearly absolutely continuous and it follows first $\eta_1$ for half of the time and then it follows $\eta_2$; it is clear
that, since $f\circ\eta_3$ coincides $\Leb{1}$-a.e. in $(0,1)$ with the function
$$
a(s):=\begin{cases} f_{\eta_1}(2st_1) & \text{ if }s\in [0,1/2], \\ f_{\eta_2}(1-2(1-s)(1-t_2)) & \text{ if }s\in [1/2,1]  \end{cases}
$$
which has a jump discontinuity at $s=1/2$, $f\circ\eta_3$ has no continuous representative. It follows that $\eta_3$ belongs to $\tilde\Gamma$ and
therefore $\int_0^1 h \circ \eta_3=\infty$. But, since 
$$ \frac 1{2t_1} \int_0^1 h \circ \eta_1 \, \d t + \frac 1{2(1-t_2)}\int_0^1 h \circ \eta_2\,\d t \geq \int_0^1 h \circ \eta_3 \,\d t$$
we get a contradiction with \eqref{eq:starstar}.

Now we define
$$ \tilde{f} (x) := \begin{cases} f_\eta(t) & \text{if $(\eta,t)\in\Theta_x$ for some $\eta\in\Gamma_g$, $t\in [0,1]$}  \\ f(x) & \text{ otherwise.} \end{cases} $$
By construction, $\tilde{f}(\eta(t))=f_\eta(t)$ for all $t\in [0,1]$ and $\eta\in\Gamma_g$, so that property (c) of $\Gamma_g$ shows \eqref{eq:Mneg}.
Using Remark~\ref{rem:implic} and the fact that $\{ \gamma\in\AC([0,1];X) \; : \; \tilde f\circ\gamma \equiv f_{\gamma} \}$ is clearly $\sim$-invariant,
we obtain (i) from \eqref{eq:Mneg}. Moreover, from \eqref{eq:Mneg} we get in particular that
\begin{equation}\label{eq:toulouse}
\Md \bigl( M (\{\gamma\in\Gamma\cap\Lipc{X}\; : \; \tilde{f}\circ\gamma\not\equiv f_\gamma \} ) \bigr) = 0.
\end{equation}
Recalling \eqref{eq:19} and the fact that $\sfj$ is a Borel isomorphism, we can rewrite \eqref{eq:toulouse} as
$$ \Md \bigl( J (\{\pgamma\in\Curvesnonpara{X}\; : \; \tilde{f} \circ \sfj \pgamma \not\equiv f_{\sfj\pgamma} \} )\bigr) =0, $$ 
and so we proved also (ii). 

\smallskip
\noindent Step 3. (The set $F:=\{f\neq\tilde f\}$ is $\mm$-negligible.) Let $\gamma^x$ be the curve identically equal $x$, that is $\gamma_t^x =x$ for all $t \in [0,1]$. It is clear that $\gamma^x$ belongs to $\Gamma$ for every $x \in X$: in particular $f_{\gamma^x}(t)= f(x)$ for every $t \in [0,1]$. The basic observation is that if we consider the set $\tilde\Gamma_c$ of constant curves $\gamma^x$ satisfying $\tilde{f} \circ \gamma^x \not\equiv f_{\gamma^x}$, then $f(x)\neq\tilde f(x)$ for every such curve, hence $\tilde\Gamma_c=\{ \gamma^x \; : \; x \in F\}$. In particular we have that $M(\tilde\Gamma_c)=\{ \delta_x \; : \; x\in F \}$.
Now, from \eqref{eq:Mneg}, we know that $\Md(M(\tilde\Gamma_c))=0$; this provides the existence of $g \in {\cal L}_+^p(X, \mm)$ 
such that $g(x)=\infty$ for every $x\in F$, and so we get that $F$ is contained in a $\mm$-negligible set.
\end{proof}
The following simple example shows that, in Theorem~\ref{lem:goodcont}, the ``nonparametric'' assumption that
$J(\AC([0,1];X)\setminus\Gamma)$ is $\Md$-negligible is not sufficient to 
conclude that $\tilde{f}=f$ $\mm$-a.e. in $X$. 

\begin{example} {\rm Let $X=[0,1]$, $\sfd$ the Euclidean distance, $\mm=\Leb{1}+\delta_{1/2}$, $p\in [1,\infty)$. The function $f$ identically 
equal to $0$ on $X\setminus\{1/2\}$ and equal to 1 at $x=1/2$ has a continuous (actually, identically equal to 0) representative $f_{\sfj\pgamma}$
for $\Md$-a.e. curve $\pgamma$, but any function $\tilde f$ such that $\tilde f\circ{\sfj\pgamma}\equiv f_{\sfj\pgamma}$ for $\Md$-a.e. $\pgamma$ should
be equal to $0$ also at $x=1/2$, so that $\mm(\{f\neq\tilde f\})=1$.
}\end{example}

Now, we are going to apply Theorem~\ref{lem:goodcont} to the problem of equivalence of Sobolev spaces. We begin with a few
preliminary results and definitions.

Let $f:X\to \R$, $g:X\to [0,\infty]$ be Borel functions. We 
consider the sets
\begin{equation}
  \label{eq:18}
  \cI(g):=\Big\{\gamma\in \AC([0,1];X):\int_\gamma g<\infty\Big\},
\end{equation}
and 
\begin{equation}
  \label{eq:16bis}
  \cB(f,g):=\Big\{\gamma\in \cI(g):
  f\circ\gamma\in W^{1,1}(0,1),\quad
  |\frac\d{\d t}(f\circ\gamma)|\le |\dot\gamma|g\circ\gamma
  \text{ $\Leb 1$-a.e.\
    in $(0,1)$}\Big\}.
\end{equation}

We will need the following simple measure theoretic lemma, which says that integration in one variable maps Borel functions to Borel functions.
Its proof is an elementary consequence of a monotone class argument (see for instance \cite[Theorem~2.12.9(iii)]{Bogachev}) and of the fact that the statement is true for $F$ bounded  and continuous.

\begin{lemma}\label{lem:souslinac}
  Let $(Y,\sfd_Y)$ be a metric space and let $F:[0,1]\times Y \to [0,\infty]$ be Borel. Then the function
  $\cII_F:Y\to [0,\infty]$ defined by $y\mapsto\int_0^1 F(t, y) \, \d t $ is a Borel function.
\end{lemma}

\begin{lemma}\label{lem:allgood}
  Let $f:X\to \R$, $g:X\to [0,\infty]$ be Borel functions. 
  Then $\cI(g)\setminus \cB(f,g)$
  is a Borel set, stable and $\sim$-invariant.
\end{lemma}
\begin{proof}
Stability is simple to check: if, by contradiction, it were $\gamma\in\cI(g)\setminus\cB(f,g)$ and 
$\sfs_{[a_n,b_n]}\gamma\in\cB(f,g)$ with $a_n\downarrow 0$
and $b_n\uparrow 1$, we would get $f\circ\gamma\in W^{1,1}(a_n,b_n)$ and
  $|\frac\d{\d t}f\circ\gamma|\le |\dot\gamma|g\circ\gamma\in L^1(0,1)$
  $\Leb 1$-a.e. in $(a_n,b_n)$. Taking limits, we would obtain $\gamma\in\cB(f,g)$, a contradiction.
  
  For the proof of $\sim$-invariance we note that, first of all, that Lemma~\ref{lem:souslinac} with $F(t,\gamma):=g(\gamma_t)|\dot\gamma_t|$
  guarantees that $\cI(g)$ is a $\sim$-invariant Borel set, provided we define $F$ using a Borel representative of $|\dot\gamma|$; this can be achieved,
  for instance, using the $\liminf$ of the metric difference quotients. Analogously, the set
  $$
  {\sf L}:=\bigl\{\gamma\in\AC([0,1];X)\;:\; \int_0^1 |f(\gamma_t)|\,\d t<\infty\bigr\}
  $$
  is Borel. Now, $\gamma \in \cB(f,g)$ if and only if $\gamma\in\cI(g)\cap\sf L$ and
  \begin{equation}\label{eqn:w11an}
   \left| \int_0^1 \phi'(t) f( \gamma_t) \, \d t \right| \leq \int_0^1 | \phi (t) | g(\gamma_t)|\dot \gamma_t | \,  \d t \qquad \text{ for all }\phi \in W
  \end{equation}
  with  $W=\{ \phi \in\AC([0,1];[0,1]) \; : \; \phi(0)=\phi(1)=0 \}$. Now, if both $\sf s$ and $\sfs^{-1}$ are absolutely continuous from
  $[0,1]$ to $[0,1]$, setting $\eta:=\gamma\circ\sfs$, we can use the change of variables formula to obtain that $(\phi\circ\sfs)' f\circ\eta\in L^1(0,1)$ 
  for all $\phi\in W$ and that
  $$
   \left| \int_0^1 (\phi \circ\sfs)'(r) f(\eta_r) \,\d r \right| \leq \int_0^1 | \phi\circ\sfs(r) | g(\eta_r)|\dot\eta_r | \,  \d r \qquad \text{ for all }\phi \in W.
  $$
  Since $W \circ\sfs = W$ we eventually obtain $\phi' f\circ\eta\in L^1(0,1)$ for all $\phi\in W$ (so that $f\circ\eta$ is locally integrable in $(0,1)$) and
  \begin{equation}\label{eqn:w11an}
   \left| \int_0^1 \phi'(r) f(\eta_r) \,\d r \right| \leq \int_0^1 | \phi(r) | g(\eta_r)|\dot\eta_r | \,  \d r \qquad \text{ for all }\phi \in W.
  \end{equation}
  It is easy to check that these two conditions, in combination with $\int_\eta g<\infty$, imply that $\eta\in\sf L$,
  therefore $f\circ\eta$ belongs to $\cB(f,g)$ and $\sim$-invariance is proved.

  In order to prove that $\cB(f,g)$ is Borel we follow a similar path: we already know that both $\cI(g)$ and $\sf L$ are
  Borel, and then in the class $\cI(g)\cap\sf L$ the condition \eqref{eqn:w11an}, now with $W$ replaced by a countable dense subset
  of $\rmC^1_c(0,1)$ for the $\rmC^1$ norm, provides a characterization of $\cB(f,g)$. Since for $\phi\in \rmC^1_c(0,1)$ fixed the maps
  $$
  \eta\in {\sf L}\mapsto \int_0^1 \phi'(r) f(\eta_r) \,\d r,\qquad\eta\mapsto\int_0^1 | \phi(r) | g(\eta_r)|\dot\eta_r | \,  \d r 
  $$
  are easily seen to be Borel in $\AC([0,1];X)$ (as a consequence of Lemma~\ref{lem:souslinac}, splitting in positive and negative part the first
  integral and using
  once more a Borel representative of $|\dot\eta|$ in the second integral)
  we obtain that $\cB(f,g)$ is Borel.
 \end{proof}

\begin{theorem}[Equivalence theorem]\label{thm:equivdeb} Any $f\in N^{1,p}(X,\sfd,\mm)$ belongs to $W^{1,p}(X,\sfd,\mm)$. Conversely,
for any $f\in W^{1,p}(X,\sfd,\mm)$ there exists a $\mm$-measurable
representative $\tilde{f}$ that belongs to $N^{1,p}(X, \sfd, \mm)$. 
More precisely, $\tilde f$ satisfies:
\begin{itemize}
\item[(i)] $\tilde f\circ\gamma\in\AC([0,1])$ for $q$-a.e. curve $\gamma\in\AC([0,1];X)$;
\item[(ii)]  $\tilde f\circ\sfj\pgamma\in\AC([0,1])$ for ${\rm Mod}_{p,\smm}$-a.e. curve $\pgamma$. 
\end{itemize}
\end{theorem}
\begin{proof} We already discussed the easy implication from $N^{1,p}$ to $W^{1,p}$, so let us focus on the converse one.
In the sequel we fix $f\in W^{1,p}(X,\sfd,\mm)$ and a $p$-weak upper gradient $g$.
By Fubini's theorem, it is easily seen that the space $W^{1,p}(X,\sfd,\mm)$ is invariant under modifications in $\mm$-negligible sets;
as a consequence, since the Borel $\sigma$-algebra is countably generated, we can assume with no loss of generality that $f$
is Borel. Another simple application of Fubini's theorem (see \cite[Remark~4.10]{Ambrosio-Gigli-Savare12}) shows that for 
$q$-a.e. curve $\gamma$ there exists
an absolutely continuous function $f_\gamma:[0,1]\to \R$ such that $f_\gamma=f\circ\gamma$ $\Leb{1}$-a.e. in $(0,1)$ and
$|\frac{\d}{\d t} f_\gamma|\leq|\dot\gamma|g\circ\gamma$ $\Leb{1}$-a.e. in $(0,1)$.
Since the $L^p$ integrability of $g$ yields that the complement of $\cI(g)$ is $q$-negligible,
we can use Lemma~\ref{lem:allgood} and Theorem~\ref{teo:popen}(ii) to infer that 
$\Sigma=\cI(g)\setminus \cB(f,g)$ satisfies  
$\Md\bigl(M(\Sigma\cap AC^\infty([0,1];X))\bigr)=0$. 

By Theorem~\ref{lem:goodcont} we obtain a $\mm$-measurable
representative $\tilde f$ of $f$ 
such that $\tilde{f}\circ\gamma\equiv f_\gamma$ for $q$-a.e. curve $\gamma$ and 
$\tilde f\circ\sfj\pgamma\equiv f_{\sfj\pgamma}$ for $\Md$-a.e. $\pgamma$. Hence, the fundamental theorem of calculus for absolutely
continuous functions gives
$$
|\tilde{f}(\pgamma_{fin})-\tilde{f}(\pgamma_{ini})|=|f_{\sfj\pgamma}(1)-f_{\sfj\pgamma} (0)|\leq\int_0^1g((\sfj\pgamma)_t)|\dot{(\sfj\pgamma)}_t|\,\d t=
\int_\pgamma g
$$
for $\Md$-a.e. $\pgamma$.
\end{proof}

\def\cprime{$'$} \def\cprime{$'$}

\end{document}